\documentclass[a4paper]{article}

\usepackage[english]{babel}
\usepackage[utf8]{inputenc}
\usepackage[margin=35mm]{geometry}
\usepackage[colorlinks,linkcolor=red,anchorcolor=green,citecolor=blue]{hyperref}
\usepackage{amsmath,amssymb,amsthm}
\usepackage{graphicx}
\usepackage{mathrsfs}
\usepackage{xcolor}

\numberwithin{equation}{section}

\newtheorem{theorem}{Theorem}[section]
\newtheorem{lemma}[theorem]{Lemma}
\newtheorem{proposition}[theorem]{Proposition}
\newtheorem{definition}[theorem]{Definition}

\newtheorem{remark}[theorem]{Remark}

\newcommand{\R}{\mathbb{R}}

\newcommand{\J}{\mathcal{J}}
\newcommand{\E}{\mathcal{E}}
\newcommand{\m}{\mathfrak{m}}

\begin{document}

\title{\textbf{Polyharmonic Almost Complex Structures}}

\author{
Weiyong He
\footnote{Department of Mathematics, University of Oregon, Eugene, OR 97403 (whe@uoregon.edu)} \qquad
\and
Ruiqi Jiang
\footnote{School of Mathematics, Hunan University, Changsha, 410082, P. R. China (jiangruiqi@hnu.edu.cn)}
}

\date{}
\maketitle

\begin{abstract}
In this paper we consider the existence and regularity of weakly polyharmonic almost complex structures on a compact almost Hermitian manifold $M^{2m}$. Such objects satisfy the elliptic system weakly $[J, \Delta^m J]=0$. We prove a very general regularity theorem for semilinear systems in critical dimensions (with \emph{critical growth nonlinearities}). In particular we prove that weakly biharmonic almost complex structures are smooth in dimension four.
\end{abstract}

\hspace*{2.7cm} critical growth of nonlinearities;


\section{Introduction}
Let $(M, g)$ be a compact Riemannian manifold of real dimension $n$, which admits a compatible almost complex structure. We denote by $\J_g$ the space of smooth almost complex structures which are compatible with $g$, i.e., $g(J\cdot, J\cdot)=g(\cdot, \cdot)$ for all $J\in \mathcal{J}_{g}$.

In recent years, the first author have studied the existence and regularity of harmonic and biharmonic almost complex structures \cite{He, He18}.
In this paper, we consider the following functional, for all $m\in \mathbb{N}^+$, $J\in\mathcal{J}_{g}$,
\begin{equation}\label{func:main}
\mathcal{E}_{m}(J)= \int_M \big|\Delta^{\frac{m}{2}} J \big|^2 dV:=
\left\{
\begin{aligned}
&\int_{M}|\nabla\Delta^{k-1}J|^{2}\,dV,  \quad m=2k-1, \\
&\int_{M}|\Delta^{k}J|^{2}\,dV,  \qquad m=2k, \\
\end{aligned}
\right.
\end{equation}
where $\nabla$ and $\Delta$ are Levi-Civita connection and Laplace-Beltrami operator on $(M, g)$ respectively, and $dV$ denotes the volume element of $(M, g)$. We call the critical points of functional $\mathcal{E}_{m}(J)$ \emph{m-harmonic almost complex structures}. These objects are tensor-valued version of polyharmonic maps which have attracted quite some attention in recent years.
Let us recall the definition of the Sobolev spaces of almost complex structures.
\begin{definition}
Suppose $(M^n,g)$ be an almost Hermitian manifold with compatible almost complex structures in $\J_g$.
We define $W^{k,p}(\J_g)$ to be the closed subspace of $W^{k,p}(T^*M \otimes TM)$ consisting of those sections $J \in W^{k,p}(T^*M \otimes TM)$ (locally $J$ is a section of $T^*M \otimes TM$ with $W^{k,p}$ coefficients), which satisfy the compatible condition almost everywhere,
\begin{align}\label{cond:main}
J^2=-id,\quad g(J\cdot, J\cdot)=g(\cdot, \cdot).
\end{align}
\end{definition}

Now we state our main results.
\begin{itemize}
\item \textbf{Theorem 3.1}: Suppose $(M^{n},g)$ is a compact almost Hermitian manifold without boundary. Then, there always exists an energy-minimizer of $\E_{m}(J)$ in $W^{m,2}(\J_g)$.

\item \textbf{Theorem 6.1}: Suppose $m \in \{ 2,3 \}$ and $J\in W^{m,2}(\mathcal{J}_{g})$ is a weakly $m$-harmonic almost complex structure on $(M^n, g)$ of $n=2m$. Then $J$ is H\"older-continuous.

\item \textbf{Theorem 7.1}: Suppose $n\geq 2m$ $(m \geq 1)$ and $J \in C^{0,\alpha} \cap W^{m,2}$ is a weakly $m$-harmonic almost complex structure on $(M^n, g)$. Then $J$ is smooth.
\end{itemize}
For semilinear elliptic systems with \emph{critical growth nonlinearities}, the most essential step towards the smoothness is to prove the H\"older continuity, such as the systems for harmonic maps, biharmonic maps and polyharmonic maps, see for example  \cite{SU, Helein, CWY, Wang, GS} and references therein.
It is well-known that a semilinear elliptic system with \emph{critical growth nonlinearities} and at \emph{critical dimension} might be singular \cite{Frehse, HJ}. For weakly harmonic map, it can be even singular everywhere \cite{Riviere} when the dimension is three and above. The smooth regularity in general starts with Helein's seminal result \cite{Helein} for harmonic maps in dimension two (the critical dimension for harmonic map) where the special (algebraic) structure of the system plays a substantial role. New proofs and understanding of Helein's seminal results have been found \cite{CWY0, Riviere2}. The methods can be generalized to fourth order elliptic system in dimension four \cite{CWY, LR}. General smooth regularity for biharmonic maps and polyharmonic maps have been obtained by \cite{Wang} and \cite{GS} respectively.

We shall briefly compare our results with the results in the theory of harmonic maps, biharmonic maps and polyharmonic maps.
Theorem \ref{thm:exist:2m} is a standard practice in calculus of variations. The main point is that in our setting, the energy-minimizer is not trivial due to its tensor-valued nature. As a comparison, such an energy-minimizer in the theory of harmonic maps is trivial: constant map. This is the main motivation for the first author to study the harmonic and biharmonic almost complex structures \cite{He, He18}, from the point of view of geometric analysis.
Our regularity results are motivated mainly by \cite{CWY} and \cite{GS} respectively. Theorem \ref{thm:J:smooth} concerns the higher regularity if H\"older regularity is assumed and our method is a modification of \cite{GS}.
For the H\"older regularity, the Coulomb gauge \cite{Wang, GS} has played an important role but it does not seem to have a counterpart in our setting. Our method is motivated by the work in \cite{CWY} and \cite{GS}.
In \cite{CWY}, the authors  explore a special divergence structure of the biharmonic system into the spheres. In our setting, the restriction for an almost complex structure is the equations $J^2=-id$ and $g(J \cdot, J\cdot)=g(\cdot, \cdot)$, which share some similarities to the restriction of maps into spheres $|u|^2=1$.
On the other hand, the tensor-valued nature makes our arguments much more complicated  (mainly due to the fact that matrix multiplication is not commutative). Nevertheless we are able to show that the elliptic system for polyharmonic almost complex structures has a desired special divergence structure when $m=2$ and $m=3$. We certainly believe that this divergence structure should hold for all  weakly polyharmonic almost complex structures.

Given the special divergence structure of the system, our argument for H\"older regularity is quite different from the method used in \cite{CWY}, but more like a generalization of \cite{GS}. We use extension of maps (almost complex structures) instead of solving boundary value problem. Moreover, our methods are very general  and work for all dimensions. Another difficulty is that our background metric is not necessarily Euclidean, while most results in the setting of polyharmonic maps (see \cite{CWY, Wang, GS} etc) only consider the background metric is Euclidean. Even though the methods for semilinear system work similarly when the background metric is not Euclidean, but the Euclidean assumption is a rather great simplification in presentation.

The paper is organized as follows. In Section \ref{sec:prelim}, we collect  some facts for Lorentz spaces and Green's functions that will be used later. In Section \ref{sec:existence}, we establish the existence of the energy-minimizers of $\mathcal{E}_{m}(J)$. In Section \ref{sec:EL}, we derive the Euler-Lagrange equation of $\mathcal{E}_{m}(J)$ and prove that the weak limit of a sequence of weakly $m$-harmonic almost complex structures in $W^{m,2}$ (with bounded $W^{m,2}$ norm) is still $m$-harmonic. In Section \ref{sec:semi-regularity}, we establish decay estimates for a class of semilinear elliptic equations in critical dimension and generalize the higher regularity result in \cite{GS} due to Gastel and Scheven. Section \ref{sec:Holder} and Section \ref{sec:smooth} are devoted to the study of H\"older regularity and smoothness of weakly $m$-harmonic almost complex structures respectively. The final section provides a detailed proof of the special nonlinear structures that $m$-harmonic almost complex structure equations admit.\\

{\bf Acknowledgement:} The first author is partly supported by NSF grant, no. 1611797. The second author is partly supported by NSFC grant, no.11901181.

\section{Preliminaries}\label{sec:prelim}
In this section, we gather some facts that will be used later.
First of all, let us denote by $G(x)=c_m \ln |x|$ the fundamental solution for $\Delta^m$ on $\R^{2m}$, where $c_m$ is a suitable constant only dependent of $m$. Then we have the following lemma that will play an important role in Section \ref{sec:semi-regularity}.

\begin{lemma}\label{lem:green}
Suppose $k \in [1,2m]$ is a positive integer and $p,q \in (1,\infty)$ satisfy
\begin{align*}
1+\frac{1}{p}= \frac{k}{2m}+\frac{1}{q}.
\end{align*}
If $f\in L^{q}(\R^{2m})$, then we have
\begin{align}\label{lem:green:ineqy}
\bigg\| \int_{\R^{2m}} \nabla^k G(x-y) f(y) dy \bigg\|_{L^p(\R^{2m})} \leq C \| f\|_{L^q(\R^{2m})}
\end{align}
where $C$ is a positive constant only dependent of $m,k,q$.
\end{lemma}
\begin{proof}
Since $\nabla^{2m}G$ is a Calder\'on-Zygmund kernel, (\ref{lem:green:ineqy}) holds for $k=2m$ and all $p=q\in (1,\infty)$.

For $k=1,\cdots,2m-1$, we have
\begin{align*}
\nabla^k G \in L^{\frac{2m}{k}, \infty}(\R^{2m})
\end{align*}
where $L^{\frac{2m}{k}, \infty}(\R^{2m})$ is a Lorentz space. By the convolution inequality for Lorentz spaces (cf. \cite{ONeil} Theorem 2.6), we deduce that
\begin{align*}
\bigg\| \int_{\R^{2m}} \nabla^k G(x-y) f(y) dy \bigg\|_{L^{p,p}(\R^{2m})}
&\leq C \| \nabla^k G\|_{L^{\frac{2m}{k},\infty}(\R^{2m})}
\| f\|_{L^{q,s}(\R^{2m})}
\end{align*}
where
\begin{align*}
1+\frac{1}{p}= \frac{k}{2m}+\frac{1}{q},\quad \frac{1}{p} \leq \frac{1}{s}.
\end{align*}
The fact that $k \in [1,2m-1]$ implies $\frac{1}{p} < \frac{1}{q}$. Thus, we can choose $s=q$. Moreover, there holds that $L^p(\R^{2m})=L^{p,p}(\R^{2m})$ for all $p\in (1,\infty)$ (cf.\cite{Ziemer} Lemma 1.8.10), which impiles (\ref{lem:green:ineqy}). The proof is complete.
\end{proof}

For more details about Lorentz spaces, we refer the readers to \cite{ONeil, Ziemer}. If the readers are concerned only with the properties of Lorentz spaces, we have gathered some useful results presented in \cite{HJ}.

\medskip
Denote by $B_1$ the unit ball of $\R^n$. Then we have the following elliptic inequality for $\Delta^m$ on $\R^n$.
\begin{lemma}\label{lem:m-ellpitic}
Suppose $v(x)\in W^{m,2}(B_1) \cap L^\infty$ and $f\in L^{\infty}(B_1)$. If $v(x)$ satisfies $\Delta^m v(x) =f(x)$ in distributional sense, we have
\begin{align}
\|v(x)\|_{L^{\infty}(B_{\frac{1}{2}})} \leq C
\bigg( \|v(x)\|_{L^1(B_1)} +\|f(x) \|_{L^\infty(B_1)} \bigg),
\end{align}
where $C$ is a positive constant only dependent of $n$
\end{lemma}
\begin{proof}
It suffices to prove above inequality by the standard elliptic theory.
\end{proof}

\section{The existence of energy-minimizer}\label{sec:existence}

In this section we will establish the existence of the energy-minimizers of the functionals $\E_{m}(J)$.

\begin{theorem}\label{thm:exist:2m}
Suppose $(M^{n},g)$ is a compact almost Hermitian manifold without boundary. Then, there always exists an energy-minimizer of $\E_{m}(J)$ in $W^{m,2}(\J_g)$.
\end{theorem}
\begin{proof}
The proof is standard in calculus of variations. For the convenience of the reader we give the detailed procedure. Firstly, let us take a minimizing sequence $J_k \in W^{m,2}(\J_g)$ such that
\begin{align*}
\inf \limits_{J \in W^{m,2}} \E_{m}(J) = \lim_{k \rightarrow \infty } \E_{m}(J_k).
\end{align*}
Since $(M,g)$ is a compact manifold without boundary, there exists a positive constant $C$ only dependent of $M$ and $m$, such that
\begin{align*}
\| J \|^2_{W^{m,2}}
\leq C \bigg( \sum_{|\alpha|=m}\|\nabla^\alpha J \|^2_{L^2} + \| J\|^2_{L^\infty} \bigg)
\leq C  \big( \E_{m}(J) + 1 \big), \quad \forall J \in W^{m,2}(\J_g)
\end{align*}
where we used the Gagliardo–Nirenberg interpolation inequality, integration by parts and the fact that $\| J\|_{L^\infty}\leq C(n) <\infty$ (see Lemma \ref{lem:L_infty_J}). Then the sequence $\{J_k\}$ is bounded in $W^{m,2}$. Hence, there exists a subsequence, still denoted by $J_k$, and $J_0 \in W^{m,2}$, such that $J_k$ converges weakly to $J_0$ in $W^{m,2}(\J_g)$ and
$$\|J_0\|_{W^{m,2}} \leq \varliminf_{k \rightarrow \infty} \|J_k\|_{W^{m,2}}.$$
On the other hand, by Kondrachov compactness, we know that $J_k$ converges strongly to $J_0$ in $W^{m-1,2}$. Thus $J_0$ satisfies (\ref{cond:main}) almost everywhere which ensures $J \in \mathcal{J}_{g}$  and there holds
$$\E_{m}(J_0) \leq \varliminf_{k \rightarrow \infty} \E_{m}(J_k).$$
It follows that $J_0$ is the energy-minimizer of the functional $\E_{m}(J)$, i.e.,
\begin{align*}
\E_{m}(J_0)=\inf \limits_{J \in W^{m,2}} \E_{m}(J),
\end{align*}
which is the desired conclusion.
\end{proof}

\section{The Euler-Lagrange equation of functional \texorpdfstring{$\mathcal{E}_{m}(J)$}{Em (J)}} \label{sec:EL}

In this section, we will derive the Euler-Lagrange equation of $\mathcal{E}_{m}(J)$ and give the definition of weakly $m$-harmonic almost complex structure. Moreover, we show that the weak limit of a sequence of $W^{m,2}$ $m$-harmonic almost complex structures with bounded $W^{m,2}$ norm is still $m$-harmonic.

\medskip

For the convenience of reader, we firstly recall some notations.
Let us denote by $T_{q}^{p}(M)$ the set of all $(p,q)$ tensor fields on $(M,g)$, $\nabla$ the Levi-Civita connection and $\Delta$ the Laplace-Beltrami operator.
It is well known that there is a natural inner product on $T_{q}^{p}(M)$, denoted
by $\left\langle \right\rangle $.
In local coordinate $\{x^{i}\}_{i=1}^{n}$, $A\in T_{q}^{p}(M)$ can
be expressed by
\[
A=A_{i_1 \cdots i_q}^{j_1 \cdots j_p}
\frac{\partial}{\partial x^{j_1}} \otimes \cdots \otimes \frac{\partial}{\partial x^{j_p}}
\otimes
dx^{i_1} \otimes \cdots \otimes dx^{i_q}.
\]
and the inner product of $A,B \in T^p_q(M)$ is
\[
\left\langle A,B\right\rangle
=A_{i_1 \cdots i_q}^{k_1 \cdots k_p}
B_{j_1 \cdots j_q}^{l_1 \cdots l_p}
g^{i_1j_1} \cdots g^{i_q j_q}
g_{k_1 l_1} \cdots g_{k_p l_p},
\]
where $g=g_{ij} dx^i \otimes dx^j$ and $(g^{ij})$ is the inverse of $(g_{ij})$.

For any $A\in T_1^1(M)$, we denote the adjoint operator of $A$ by $A^*$, which is defined by
\[
g(X,A^{*}Y) := g(AX,Y),\quad \forall X,Y\in \mathfrak{X}(M).
\]
where $\mathfrak{X}(M)$ is the set of all smooth vector fields on $(M,g)$. Hence $A^* \in T^1_1(M)$ and in local coordinate $\{x^{i}\}_{i=1}^{n}$, we have
\begin{align*}
(A^*)_i^j=A_k^l g^{kj} g_{li}
\end{align*}
where $A=A_{i}^{j}\partial_{x^{j}}\otimes dx^{i}$.

\medskip

We gather some useful facts to derive the Euler-Lagrange equation of $\mathcal{E}_{m}(J)$.
\begin{proposition}\label{prop:inner}
Suppose $(M,g)$ is a compact Riemannian manifold without boundary. Then we have
\begin{enumerate}
\item For all $A,B\in T_{q}^{p}(M)$, there holds
\[
\int_{M}\left\langle \nabla A,\nabla B\right\rangle =-\int_{M}\left\langle A,\Delta B\right\rangle .
\]
\item For all $A\in T_{1}^{1}(M)$ and $X \in \mathfrak{X}(M)$, there holds
\[
\left(\nabla_X A\right)^{*}=\nabla_X (A^{*}).
\]
\item For all $A,B\in T_{1}^{1}(M)$, there holds
\[
\left\langle A,B\right\rangle =\left\langle A^{*},B^{*}\right\rangle .
\]
\item For all $A,B,C\in T_{1}^{1}(M)$, there holds
\[
\left\langle A,BC\right\rangle =\left\langle B^{*}A,C\right\rangle =\left\langle AC^{*},B\right\rangle .
\]
\end{enumerate}
\end{proposition}

\begin{proof}
These are straightforward computations.  The details are left to the reader.
\end{proof}

Note that, for any $A \in T^1_1(M)$, $A(p)$ for $p\in M$ is just a linear map on the tangent space $T_pM$. Hence, for $A, B \in T^1_1(M)$, $AB$ is regarded as the composition of linear maps, i.e., $AB\in T^1_1(M)$. More precisely, in local coordinate $\{x^{i}\}_{i=1}^{n}$, we have
\begin{align*}
(AB)_i^j=A_s^jB_i^s,
\end{align*}
where $A=A_s^j \partial_{x^j} \otimes dx^s$, $B=B_i^s \partial_{x^s} \otimes dx^i$ and $AB=(AB)_i^j \partial_{x^j} \otimes dx^i$.

As an application of Proposition \ref{prop:inner}, we can obtain the $L^\infty$-norm of $J$ in $\J_g$.
\begin{lemma}\label{lem:L_infty_J}
If $J \in \J_g$, then $\| J \|_{L^\infty}=\sqrt{n}$ where $n$ denotes the dimension of $M$.
\end{lemma}
\begin{proof}
By definition of $\J_g$, for all $J\in \J_g$, there holds
\begin{align*}
J^2=-id, \quad
g(JX,JY)=g(X,Y),\,\,\forall X,Y \in \mathfrak{X}(M).
\end{align*}
Thus, for all $X,Y \in \mathfrak{X}(M)$, we have
\begin{align*}
g(X,JY)+g(X,J^*Y)
&=g(X,JY)+g(JX,Y) \\
&=g(X,JY)+g(J^2X,JY)\\
&=g(X,JY)-g(X,JY)\\
&=0,
\end{align*}
which implies
\begin{align*}
J+J^*=0.
\end{align*}
Hence, the condition (\ref{cond:main}) is clearly equivalent to
\begin{align}\label{cond:main2}
J^2=-id, \quad J+J^*=0.
\end{align}
Then we apply (\ref{cond:main2}) and Proposition \ref{prop:inner} to obtain
\begin{align*}
\langle J,J\rangle
=\langle -J^*,J\rangle
=-\langle J^*id,J\rangle
=-\langle id,(J^*)^*J\rangle
=-\langle id,JJ\rangle
=\langle id, id\rangle
=n.
\end{align*}
Hence, $\| J \|_{L^\infty}=\sqrt{n}$.
\end{proof}

\medskip
Suppose the family $\{ J(t)\}_{t \in (-\delta ,\delta)}$ is an admissible variation of $J$ in the space $\mathcal{J}_{g}$, i.e., $J(0)=J$ and $J(t) \in \mathcal{J}_{g}$ for all $t\in (-\delta ,\delta)$. If $S=\frac{d}{dt} J(t)|_{t=0}$ exists, $S$ is called \emph{an admissible variational direction} of $J$ in $\mathcal{J}_{g}$. Let us define $\mathcal{S}_J$ to be the collection of all admissible variational directions of $J$ in $\mathcal{J}_{g}$.

\begin{proposition}
Suppose $(M,g)$ is a compact almost Hermitian manifold without boundary and $\mathcal{J}_{g}$ the space of all smooth almost complex structures on $M$ which are compatible with $g$. Then for any fix $J\in \mathcal{J}_{g}$, we have
\begin{align}
\mathcal{S}_J=\{S \in T_1^1(M): SJ+JS=0, S+S^*=0 \}.
\end{align}
\end{proposition}

\begin{proof}
It follows from (\ref{cond:main2}) that
\begin{align}\label{cond:J}
\mathcal{J}_{g}=\{J\in T^1_1(M):\, J^2=-id, \, J^*+J=0\}.
\end{align}

Suppose $\{ J(t)\}_{t \in (-\delta ,\delta)}$ is an admissible variation of $J$ in the space $\mathcal{J}_{g}$. Then for all $t\in (-\delta ,\delta)$, there holds
\begin{align*}
J(t)^2=-id, \,\, J(t)^*+J(t)=0.
\end{align*}
By direct computation, we have
\begin{align}\label{identity:S}
SJ+JS=0, \, S^*+S=0.
\end{align}
where $S=\frac{d}{dt} J(t)|_{t=0}$. On the other hand, for any $S\in T^1_1(M)$ satisfies property (\ref{identity:S}), we can choose
\begin{align*}
J(t)=J\exp \{ tSJ\},\quad \forall t \in (-\delta, \delta)
\end{align*}
which is well-defined for $\delta>0$ small enough due to the continuousness of $S,J$ on $M$ and the compactness of $M$. It is easy to check that $J(t)\in \mathcal{J}_{g}$ for all $t \in (-\delta, \delta)$ and
\begin{align*}
\frac{d}{dt} J(t)|_{t=0}=JSJ=S.
\end{align*}

Hence, we have
\begin{align}
\mathcal{S}_J=\{S \in T_1^1(M): SJ+JS=0, S+S^*=0 \}.
\end{align}
\end{proof}

For any $J\in \mathcal{J}_{g}$, we define an operator $\Phi_J: T^1_1(M) \rightarrow S_J$ by
\begin{align}
\Phi_J(T)=\frac{1}{4} \bigg( (T+JTJ)-(T+JTJ)^* \bigg).
\end{align}
It is a simple matter to check that $\Phi_J$ is a surjective linear map, i.e., $\Phi_J(T^1_1(M))=S_J$. Moreover, $\Phi_J|_{S_J}=id$.

\medskip
We are now in a position to derive the Euler-Lagrange equation of functional $\mathcal{E}_{m}(J)$. Suppose $J\in \mathcal{J}_{g}$ is a critical point of $\mathcal{E}_{m}(J)$. Then for any $S\in \mathcal{S}_J$, we have
\begin{align}\label{eqn:1-variation}
0=\delta \mathcal{E}_{m}(J)
=(-1)^m 2\int_M \left\langle \Delta^m J, S \right\rangle.
\end{align}
Since
\begin{align*}
\left\langle \Phi_J(\Delta^m J), S \right\rangle
=&\frac{1}{4}\big\langle
(\Delta^m J+ J\Delta^m J J)-  (\Delta^m J+ J\Delta^m J J)^*, S \big\rangle \\
=&\frac{1}{4}\big\langle \Delta^m J+ J\Delta^m J J, S \big\rangle
- \frac{1}{4}\big\langle(\Delta^m J+ J\Delta^m J J)^*, S \big\rangle \\
=&\frac{1}{4}\big\langle \Delta^m J+ J\Delta^m J J, S \big\rangle
- \frac{1}{4}\big\langle \Delta^m J+ J\Delta^m J J, S^* \big\rangle \\
=&\frac{1}{2}\big\langle \Delta^m J+ J\Delta^m J J, S \big\rangle \\
=&\frac{1}{2}\big\langle \Delta^m J, S \big\rangle
+ \frac{1}{2}\big\langle J\Delta^m J J, S \big\rangle\\
=&\frac{1}{2}\left\langle \Delta^m J, S \right\rangle
+ \frac{1}{2}\left\langle \Delta^m J , J^*SJ^* \right\rangle\\
=&\left\langle \Delta^m J, S \right\rangle,
\end{align*}
we deduce that, for all $S\in \mathcal{S}_J$, there holds
\begin{align*}
\int_M \left\langle \Phi_J(\Delta^m J), S \right\rangle=0
\end{align*}
which implies that
\begin{align}\label{eqn:EL-temp}
\Phi_J(\Delta^m J)=0.
\end{align}
The fact that $(\Delta^m J+ J\Delta^m J J)^*=\Delta^m J^*+ J^*\Delta^m J^* J^*=-(\Delta^m J+ J\Delta^m J J)$ yields an equivalent form of (\ref{eqn:EL-temp}), which reads
\begin{align}\label{eqn:EL-form 1}
\Delta^m J +J \Delta^m J J=0.
\end{align}
For simplicity, we can rewrite (\ref{eqn:EL-form 1}) in a simple form
\begin{align*}
[\Delta^m J, J]:=\Delta^m J J-J \Delta^m J=0.
\end{align*}

\begin{theorem}\label{thm:EL-smooth}
The Euler-Lagrange equation of functional $\mathcal{E}_{m}(J)$ is
\begin{align}\label{eqn:EL-form 2}
[\Delta^m J, J]=0.
\end{align}
\end{theorem}

An almost complex structure $J\in W^{m,2}(\mathcal{J}_{g})$ is called a \emph{weakly $m$-harmonic almost complex structure} if it satisfies (\ref{eqn:EL-form 2}) in distributional sense.
For the further study of $m$-harmonic almost complex structures, we will deduce two equivalent forms of (\ref{eqn:EL-form 2}) in distributional sense, which are stated in Proposition \ref{prop:EL-div} and \ref{prop:EL-weaklim}.

\begin{proposition}\label{prop:EL-div}
A critical point of $\mathcal{E}_{m}(J)$ on $W^{m,2}(\mathcal{J}_{g})$ satisfies the following equation in distributional sense,
\begin{align}\label{eqn:EL-weak}
\Delta^m J=\sum_{s=0}^{m-1}(-1)^{m+1+s} \nabla^s \cdot g_s
\end{align}
where
\begin{align}
g_s =\sum_{\substack {k_1+k_2+k_3=2m-s \\ 0 \leq k_1, k_2, k_3 \leq m}}
C(k_1, k_2, k_3)\,
\nabla^{k_1} J \nabla^{k_2} J \nabla^{k_3} J,
\quad C(k_1,k_2,k_3)\in \mathbb{Z}.
\end{align}
More precisely, for any $T \in C^\infty_0(T^1_1(M))$ the space of all smooth (1,1) tensor fields on $M$ with compact support, there holds
\begin{enumerate}
\item when $m=2k$, $k\in \mathbb{N}^+$,
\begin{align}\label{prop:EL-weak:m=2k}
\int_M \left\langle \Delta^k J , \Delta^k T \right\rangle
+\sum_{s=0}^{m-1} \int_M \left\langle g_s, \nabla^s T \right\rangle =0
\end{align}
\item when $m=2k-1$, $k\in \mathbb{N}^+$,
\begin{align}\label{prop:EL-weak:m=2k-1}
\int_M \left\langle \nabla \Delta^{k-1} J , \nabla \Delta^{k-1} T \right\rangle
+\sum_{s=0}^{m-1} \int_M \left\langle g_s, \nabla^s T \right\rangle =0
\end{align}
\end{enumerate}
For simplicity, we will give the exact meaning of $\nabla^s$ in the proof.

\end{proposition}

\begin{proof}
In the sequel, we will focus on the case $m=2k$ for $k\in \mathbb{N}^+$. Because, one can take the similar process to obtain the result for case $m=2k-1$, $k\in \mathbb{N}^+$.

Suppose $J\in \mathcal{J}_{g}$ is a critical point of $\mathcal{E}_{m}(J)$. Then for any $S\in \mathcal{S}_J$, we have
\begin{align*}
0=2\int_M \left\langle \Delta^k J, \Delta^k S \right\rangle.
\end{align*}
Since $\Phi_J(T)\in S_J$ for any $T \in C^\infty_0(T^1_1(M))$, we have
\begin{align}
0=& \frac{1}{2} \int_M \bigg\langle \Delta^k J \, ,\, \Delta^k \big((T+JTJ)-(T+JTJ)^* \big) \bigg\rangle \notag\\
=&\frac{1}{2} \int_M \bigg\langle \Delta^k J, \Delta^k \big(T+JTJ \big) \bigg\rangle
-\bigg\langle \Delta^k J, \Delta^k \big(T+JTJ \big)^* \bigg\rangle \notag\\
=&\frac{1}{2} \int_M \bigg\langle \Delta^k J, \Delta^k \big(T+JTJ \big) \bigg\rangle
-\bigg\langle \Delta^k J^*, \Delta^k \big(T+JTJ \big) \bigg\rangle \notag\\
=&\int_M \bigg\langle \Delta^k J, \Delta^k \big(T+JTJ \big) \bigg\rangle \label{prop:EL-weak:temp2}\\
=&\int_M \left\langle \Delta^k J, \Delta^k T \right\rangle
+\left\langle \Delta^k J, J\Delta^k TJ \right\rangle
+\left\langle \Delta^k J, R_1 \right\rangle \notag\\
=&\int_M \left\langle \Delta^k J, \Delta^k T \right\rangle
+\left\langle J^* \Delta^k J J^*, \Delta^k T \right\rangle
+\left\langle \Delta^k J, R_1 \right\rangle \notag\\
=&\int_M \left\langle \Delta^k J, \Delta^k T \right\rangle
+\left\langle J \Delta^k J J, \Delta^k T \right\rangle
+\left\langle \Delta^k J, R_1 \right\rangle \notag\\
=&\int_M \left\langle \Delta^k J, \Delta^k T \right\rangle
-\left\langle \big(\Delta^k J J+R_2 \big) J, \Delta^k T \right\rangle
+\left\langle \Delta^k J, R_1 \right\rangle \notag\\
=&\int_M 2\left\langle \Delta^k J, \Delta^k T \right\rangle
-\left\langle R_2 J, \Delta^k T \right\rangle
+\left\langle \Delta^k J, R_1 \right\rangle \notag\\
=&\int_M 2\left\langle \Delta^k J, \Delta^k T \right\rangle
+\left\langle \nabla (R_2 J), \nabla \Delta^{k-1} T \right\rangle
+\left\langle \Delta^k J, R_1 \right\rangle \label{prop:EL-weak:temp1}
\end{align}
where
\begin{align*}
R_1= \Delta^k (JTJ)-J \Delta^k T J,
\end{align*}
and due to $J^2 =- id$,
\begin{align*}
0=\Delta^k (JJ)=\Delta^k J J+ J\Delta^k J +R_2.
\end{align*}
In order to describe the terms of $R_1$ and $R_2$, we will give their local expression by taking a local orthonormal fields $\{ e_i\}_{i=1}^n$ on $M$. For simplicity, we use the notation $\nabla_i:=\nabla_{e_i}$ and make the convention that the same indices mean summation. Hence, $\Delta=\nabla_i^2=\nabla_i \nabla_i$. Let us begin to compute the terms $R_1$ and $R_2$,
\begin{align*}
R_1=&\Delta^k (JTJ) -J \Delta^k T J \\
=&\nabla_{i_1}^2 \cdots \nabla_{i_k}^2 (JTJ) -J \Delta^k T J \\
=&\sum_{\substack{\alpha<, \beta<, \gamma< \\
k_1+k_2+k_3=m\\
k_2\leq m-1
}}
\nabla_{i_{\alpha_1}} \cdots \nabla_{i_{\alpha_{k_1}}} J \,
\nabla_{i_{\beta_1}} \cdots \nabla_{i_{\beta_{k_2}}} T\,
\nabla_{i_{\gamma_1}} \cdots \nabla_{i_{\gamma_{k_3}}} J
\end{align*}
and
\begin{align*}
R_2=&\Delta^k (JJ)-\Delta^k J J-J\Delta^k J\\
=&\sum_{\substack{\alpha<, \beta<, \\
k_1+k_2=m\\
1\leq k_1, k_2\leq m-1
}}
\nabla_{i_{\alpha_1}} \cdots \nabla_{i_{\alpha_{k_1}}} J \,
\nabla_{i_{\beta_1}} \cdots \nabla_{i_{\beta_{k_2}}} J
\end{align*}
where the symbol $\alpha<$ means $1\leq \alpha_1\leq \cdots \leq \alpha_{k_1}\leq k$, so do $\beta<$ and $\gamma<$. To simplify notation, we define
\begin{align*}
\nabla^{k_1} = \sum_{\alpha<} \nabla_{i_{\alpha_1}} \cdots \nabla_{i_{\alpha_{k_1}}} .
\end{align*}
Then we can rewrite $R_1$ and $R_2$ in the following simple form
\begin{align*}
R_1=&\sum_{\substack{k_1+k_2+k_3=m \\k_2\leq m-1}}
\nabla^{k_1} J \,\nabla^{k_2} T\,\nabla^{k_3} J \\
R_2=&\sum_{\substack{k_1+k_2=m\\ 1\leq k_1, k_2\leq m-1}}
\nabla^{k_1} J \,\nabla^{k_2} J.
\end{align*}
Substituting above equalities into (\ref{prop:EL-weak:temp1}), we have
\begin{align*}
0=&\int_M 2\left\langle \Delta^k J, \Delta^k T \right\rangle
+\int_M\sum_{\substack{k_1+k_2=m\\ 1\leq k_1, k_2\leq m-1}}
\left\langle \nabla (\nabla^{k_1} J \,\nabla^{k_2} J J), \nabla \Delta^{k-1} T \right\rangle \\
&+\int_M \sum_{\substack{k_1+k_2+k_3=m \\k_2\leq m-1}}
\left\langle \Delta^k J, \nabla^{k_1} J \,\nabla^{k_2} T\,\nabla^{k_3} J  \right\rangle\\
=&\int_M 2\left\langle \Delta^k J, \Delta^k T \right\rangle
+\int_M\sum_{\substack{k_1+k_2=m\\ 1\leq k_1, k_2\leq m-1}}
\left\langle \nabla (\nabla^{k_1} J \,\nabla^{k_2} J J), \nabla \Delta^{k-1} T \right\rangle \\
&+ \int_M\sum_{\substack{k_1+k_2+k_3=m \\k_2\leq m-1}}
\left\langle \nabla^{k_1} J \Delta^k J \,\nabla^{k_3} J,  \nabla^{k_2} T  \right\rangle,
\end{align*}
which implies (\ref{prop:EL-weak:m=2k}).
\end{proof}

It follows from the argument in above proof that (\ref{eqn:EL-weak}) is equivalent to  (\ref{eqn:EL-form 2}) for $J\in W^{m,2}(\mathcal{J}_{g})$ in distributional sense.

\begin{proposition}\label{prop:EL-weaklim}
A weakly $m$-harmonic almost complex structure $J\in W^{m,2}(\mathcal{J}_{g})$ satisfies the following, for all $T \in C^\infty_0(T^1_1(M))$,
\begin{enumerate}
\item when $m=2k$, $k\in \mathbb{N}^+$
\begin{align}\label{prop:EL-weak-2:m=2k}
\int_M \left\langle \Delta^k J, [J, \Delta^k T]\right\rangle
+\sum_{\substack{k_1+k_2=m \\ 1\leq k_1, k_2 \leq m-1}}
\left\langle \Delta^k J, [\nabla^{k_1} J, \nabla^{k_2}T]\right\rangle =0,
\end{align}
\item when $m=2k-1$, $k\in \mathbb{N}^+$
\begin{align}\label{prop:EL-weak-2:m=2k-1}
\int_M \left\langle \nabla \Delta^{k-1} J, [J, \nabla \Delta^{k-1} T]\right\rangle +\sum_{\substack{k_1+k_2=m \\ 1\leq k_1, k_2 \leq m-1}}
\left\langle \nabla \Delta^{k-1} J, [\nabla^{k_1} J, \nabla^{k_2}T]\right\rangle =0.
\end{align}
\end{enumerate}
Moreover, the weak limit of a sequence of weakly $m$-harmonic almost complex structures in $W^{m,2}$ (with bounded $W^{m,2}$ norm) is still $m$-harmonic.
\end{proposition}

\begin{proof}
Here we only prove the results in the case $m=2k$, $k\in \mathbb{N}^+$. A similar argument can yield the conclusion in the case $m$ is odd.

For a weakly $m$-harmonic almost complex structure $J \in W^{m,2}(\mathcal{J}_{g})$, we can recall (\ref{prop:EL-weak:temp2}), i.e., for all $T \in C^\infty_0(T^1_1(M))$, there holds
\begin{align*}
\int_M \left\langle \Delta^k J, \Delta^k \big(T+JTJ \big) \right\rangle =0
\end{align*}
By replacing $T$ by $JT$, we deduce that, for all $T \in C^\infty_0(T^1_1(M))$, there holds
\begin{align*}
\int_M \left\langle \Delta^k J, \Delta^k \big(JT-TJ \big) \right\rangle
=\int_M \left\langle \Delta^k J, \Delta^k [J,T] \right\rangle =0.
\end{align*}
Since
\begin{align*}
\Delta^k [J,T]=[\Delta^k J, T]+[J,\Delta^k T]
+\sum_{\substack{k_1+k_2=m \\ 1\leq k_1, k_2 \leq m-1}}
[\nabla^{k_1}J, \nabla^{k_2}T],
\end{align*}
and
\begin{align*}
\left\langle \Delta^k J, [\Delta^k J, T] \right\rangle
=&\left\langle \Delta^k J, \Delta^k J\, T \right\rangle
-\left\langle \Delta^k J,  T \,\Delta^k J \right\rangle\\
=&\left\langle (\Delta^k J)^*\,\Delta^k J,  T \right\rangle
-\left\langle \Delta^k J\,(\Delta^k J)^*,  T  \right\rangle\\
=&\left\langle \Delta^k J^*\,\Delta^k J,  T \right\rangle
-\left\langle \Delta^k J\,\Delta^k J^*,  T  \right\rangle\\
=&0,
\end{align*}
(\ref{prop:EL-weak-2:m=2k}) follows immediately. It is clear that (\ref{prop:EL-weak-2:m=2k}) is equivalent to  (\ref{eqn:EL-form 2}) for $J\in W^{m,2}(\mathcal{J}_{g})$ in distributional sense.

Now we will apply (\ref{prop:EL-weak-2:m=2k}) to prove that the weak limit of a sequence of weakly $m$-harmonic almost complex structures in $W^{m,2}$ (with bounded $W^{m,2}$ norm) is still $m$-harmonic. Suppose $\{J_l\}$ is a sequence of weakly $m$-harmonic almost complex structures in $W^{m,2}$ such that
\begin{align*}
J_l \rightharpoonup J_0 \quad \mbox{in}\,\, W^{m,2}
\quad \mbox{and} \quad
\sup_l \|J_l\|_{W^{m,2}} <\infty.
\end{align*}
By Rellich–Kondrachov theorem, we know that $J_l$ converges to $J_0$ in $W^{m-1,2}$. Hence $J_0 \in W^{m,2}(\mathcal{J}_{g})$.
Since $J_l \rightharpoonup J_0$ in $W^{m,2}$, we have
\begin{align}\label{prop:EL-weak-2:temp1}
\lim_{l\rightarrow \infty}
\int_M \left\langle \Delta^k J_l, [J_0, \Delta^k T]\right\rangle
= \int_M \left\langle \Delta^k J_0, [J_0, \Delta^k T]\right\rangle.
\end{align}
Since $J_l$ converges to $J_0$ in $W^{m-1,2}$, $\sup_l \|J_l\|_{W^{m,2}} <\infty$  and
\begin{align*}
\left|\int_M \left\langle \Delta^k J_l, [J_l-J_0, \Delta^k T]\right\rangle \right|
\leq \|\Delta^k J_l\|_{L^2} \|J_l-J_0\|_{L^2} \|\Delta^k T\|_{L^\infty},
\end{align*}
we have
\begin{align}\label{prop:EL-weak-2:temp2}
\lim_{l\rightarrow \infty}
\int_M \left\langle \Delta^k J_l, [J_l-J_0, \Delta^k T]\right\rangle
= 0.
\end{align}
Combining (\ref{prop:EL-weak-2:temp1}) and (\ref{prop:EL-weak-2:temp2}), we deduce that
\begin{align}
\lim_{l\rightarrow \infty}
\int_M \left\langle \Delta^k J_l, [J_l, \Delta^k T]\right\rangle
=\int_M \left\langle \Delta^k J_0, [J_0, \Delta^k T]\right\rangle.
\end{align}
Similarly we conclude that, for all $k_1+k_2=m$ and $1\leq k_1, k_2 \leq m-1$,
\begin{align}
\lim_{l\rightarrow \infty}
\int_M \left\langle \Delta^k J_l, [\nabla^{k_1} J_l, \nabla^{k_2}T]\right\rangle
=\int_M \left\langle \Delta^k J_0, [\nabla^{k_1} J_0, \nabla^{k_2}T]\right\rangle.
\end{align}
Hence, $J_0$ is also a weakly $m$-harmonic almost complex structure and the proof is complete.
\end{proof}

\section{The regularity of a class of semilinear elliptic equations}\label{sec:semi-regularity}

In this section, we will establish decay estimates for a class of semilinear elliptic equations in critical dimension and generalize the regularity results in \cite{GS} due to Gastel and Scheven.

\medskip
Suppose $B_1$ is a unit ball in $\mathbb{R}^n$ centered at origin.
Let us consider the following semilinear elliptic equation for $u: B_1 \subset \mathbb{R}^n\rightarrow \R^K$, $K\in \mathbb{N}^+$,
\begin{align}\label{eqn:semilinear}
\Delta^m u = \Psi (x, \nabla u, \cdots, \nabla^{2m-1} u)
\end{align}
where $\Psi: \R^n \times \R^{nK}\times \cdots \times \R^{n^{2m-1}K} \rightarrow \R^K$ is smooth.
\subsection{Decay estimates for \texorpdfstring{$W^{2,2}$}{W22} biharmonic almost complex structure on \texorpdfstring{$B_1 \subset \R^4$}{B1}} \label{sec:decay_biharm}
In order to illustrate the main idea of our proof of decay estimates for a class of semilinear elliptic equations in the next subsection, first we consider a special case: decay estimates for biharmonic almost complex structure defined on the unit ball $B_1$ in Euclidean space $\R^4$. The presentation is clearer and more streamlined for this case and the main ideas are essentially the same.
More precisely, let us consider the biharmonic almost complex structure equation
\begin{equation}
\Delta^2 J = J \bigg(  \nabla \Delta J \nabla J  + \nabla J \nabla \Delta J+ \Delta J \Delta J
+  \Delta( \nabla J )^2\bigg)
\end{equation}
where $J : B_1 \subset \R^4 \rightarrow M_4(\R)$ ($M_4(\R)$ the set of all $4\times 4$ real matrices) satisfies
\begin{equation}\label{cond:J:euclidean}
J^2=- id, \quad J+J^T=0
\end{equation}
where $id$ denotes the identity matrix and $J^T$ denotes the transpose of $J$. Note that the condition (\ref{cond:J:euclidean}) is just (\ref{cond:J}) when $(M, g)$ is the unit ball in Euclidean space. It follows from Proposition \ref{prop:T_lambda} that for any given constant matrix $\lambda_0$, biharmonic almost complex structure $J$ always satisfies
\begin{align}\label{eqn:biharm:lambda}
\Delta^2 J = T_{\lambda_0}
\end{align}
where $T_{\lambda_0}$ is a linear combination of the following terms
\begin{align*}
\nabla^{\alpha}\bigg((J-\lambda_{0})\ast\nabla^{\beta}J\ast\nabla^{\gamma}J\bigg)
\quad or\quad
\lambda_0\ast\nabla^{\alpha}\bigg((J-\lambda_{0})\ast\nabla^{\delta}J\bigg),
\end{align*}
where $\alpha,\beta,\gamma,\delta$ are multi-indices such that $1\leq|\alpha| \leq 3$, $0\leq |\beta|,|\gamma|,|\delta|\leq 2$, $|\alpha|+|\beta|+|\gamma|=4$
and $|\alpha|+|\delta|=4$. The notation $A*B$ means the composition of terms $A$ and $B$, such as $AB$ and $BA$. Then we have the following decay estimate.
\begin{lemma}\label{lem:decay_biharmonic}
Suppose $J \in W^{2,2}(B_1, M_4(\R))$ is a weakly biharmonic almost complex structure on unit ball $B_1 \subset \R^4$. Then, given any $\tau \in (0,1)$, there exists $\epsilon_0>0$ and $\theta_0 \in (0, \frac{1}{2})$ such that if
\begin{align}
E(J,1):=\bigg( \int_{B_1} |\nabla J|^4 \bigg)^\frac{1}{4} + \bigg( \int_{B_1} |\nabla^2 J|^2 \bigg)^\frac{1}{2}
\leq \epsilon_0,
\end{align}
then we have
\begin{align}\label{lem:S:decay}
D_{p_0}(J,\theta_0) \leq \theta_0^\tau D_{p_0}(J,1),
\end{align}
where $p_0=\frac{8}{3}$ and
\begin{align}
D_{p}(J,r):= \bigg(r^{p-4} \int_{B_r} |\nabla u|^p \bigg)^{\frac{1}{p}}.
\end{align}
\end{lemma}
\begin{proof}

Firstly, by a simple calculation, we obtain
\begin{align}
|J|^2=\langle J, J \rangle
=\sum_{k,l} (J^k_l)^2=\sum_{k,l}(-J^k_l J^l_k)=4
\end{align}
where we used the condition (\ref{cond:J:euclidean}) in the last two equalities. Hence, $\|J\|_{L^\infty}=2<\infty$. For simplicity, we always denote by $C$ the positive constant independent of $J$.

Secondly, we need to extend $J$ to $\widetilde{J}\in W^{2,2}(\R^4, M_4(\R)) \cap L^\infty$ such that
\begin{align*}
\widetilde{J}|_{B_1}=J, \quad \widetilde{J}|_{\mathbb{R}^{4} \setminus B_2}= \lambda_0
\end{align*}
where $\lambda_0=\frac{1}{|B_1|}\int_{B_1} J$, and
\begin{align}
\| \widetilde{J}\|_{L^\infty(\mathbb{R}^{4})}
\leq& C\, \|J \|_{L^\infty(B_1)} \label{lem:S:ineqy:temp1} \\
\| \nabla \widetilde{J}\|_{L^{p_0}(\mathbb{R}^{4})}
\leq& C\, \| \nabla J\|_{L^{p_0}(B_1)} \label{lem:S:ineqy:temp2}\\
E(\widetilde{J},\infty) \leq& C\, E(J,1),\label{lem:S:ineqy:temp3}
\end{align}
where $p_0=\frac{8}{3} \in (1, 4)$. Now, let us show how to find such an extension $\widetilde{J}$. By applying the standard extension theorem to $J-\lambda_0$ in $B_1$, we deduce that there exists a function $\widetilde{J}-\lambda_0$ defined on $\mathbb{R}^{4}$ which has a compact support contained in $B_2$ and satisfies
\begin{align}
\|\widetilde{J}-\lambda_0 \|_{L^\infty(\mathbb{R}^{4})}
\leq& C \|J -\lambda_0\|_{L^\infty(B_1)}, \label{lem:S:ineqy:temp4}\\
\|\widetilde{J} -\lambda_0\|_{W^{1,p_0}(\mathbb{R}^{4})}
\leq& C \|J -\lambda_0 \|_{W^{1,p_0}(B_1)} \label{lem:S:ineqy:temp5}\\
\| \widetilde{J} -\lambda_0 \|_{W^{2,2}(\mathbb{R}^{4})}
\leq& C \| J- \lambda_0\|_{W^{2,2}(B_1)}. \label{lem:S:ineqy:temp6}
\end{align}
Note that, since $\widetilde{J}-\lambda_0$ has a compact support, (\ref{lem:S:ineqy:temp4}) implies $\widetilde{J}-\lambda_0 \in L^q(\mathbb{R}^{4})$ for all $q\in [1,\infty]$. Then (\ref{lem:S:ineqy:temp1}) follows immediately from (\ref{lem:S:ineqy:temp4}) and a simple application of Poinc\'are inequality to the right-hand side of (\ref{lem:S:ineqy:temp5}) yields (\ref{lem:S:ineqy:temp2}).
Since $\widetilde{J}-\lambda_0$ has a compact support, by Poinc\'are inequality, Sobolev inequality and (\ref{lem:S:ineqy:temp6}), we have
\begin{align*}
E(\widetilde{J},\infty)
=& E(\widetilde{J}-\lambda_0,\infty) \\
\leq& C \| \nabla^2 \widetilde{J} \|_{L^2(\mathbb{R}^{4})}
\leq C \| J- \lambda_0\|_{W^{2,2}(B_1)}
\leq C \| \nabla J\|_{W^{1,2}(B_1)}
\leq C E(J,1)
\end{align*}
which implies (\ref{lem:S:ineqy:temp3}).
Note that $\widetilde{J}$ may not satisfy the condition (\ref{cond:J:euclidean}) outside the unit ball $B_1$.

Thirdly, denote $G(x)=c \ln |x|$ to be the fundamental solution for $\Delta^2$ on $\mathbb{R}^{4}$, where $c$ is a constant. Then $\nabla^{4} G$ is a Calder\'on-Zygmund kernel. Let us define
\begin{align*}
\omega(x)
=&\int_{\mathbb{R}^{4}} G(x-y)
\widetilde{T}_{\lambda_0} (y) dy \\
=&\sum_{\alpha, \beta, \gamma}
\omega_{\alpha,\beta,\gamma}:= \int_{\R^4} \nabla^{\alpha}G(x-y)
\bigg( \big(\widetilde{J}(y)-\lambda_0 \big) \ast \nabla^\beta \widetilde{J}(y) \ast \nabla^\gamma \widetilde{J}(y) \bigg) dy\\
& + \sum_{\alpha, \delta}
\omega_{\alpha,\delta}:= \int_{\R^4} \nabla^{\alpha}G(x-y)
\bigg( \lambda_0 \ast \big(\widetilde{J}(y)-\lambda_0 \big) \ast \nabla^\delta \widetilde{J}(y) \bigg) dy
\end{align*}
where $\widetilde{T}_{\lambda_0}$ is defined by replacing $J$ by $\widetilde{J}$ in $T_{\lambda_0}$ (see (\ref{eqn:biharm:lambda})) , and $\alpha,\beta,\gamma,\delta$ are multi-indices such that $1\leq|\alpha| \leq 3$, $0\leq |\beta|,|\gamma|,|\delta|\leq 2$, $|\alpha|+|\beta|+|\gamma|=4$
and $|\alpha|+|\delta|=4$.
We claim that for $E(J,1)\leq 1$, there holds
\begin{align}\label{lem:S:decay:main}
\| \nabla \omega\|_{L^{p_0}(B_1)}
\leq C E(J,1) \| \nabla J \|_{L^{p_0}(B_1)}.
\end{align}
We will prove above inequality term by term. Firstly, we estimate the terms $\omega_{\alpha,\beta, \gamma}$. By Lemma \ref{lem:green}, we have
\begin{align*}
\|\nabla \omega_{\alpha,\beta, \gamma} \|_{L^{p_0}(\R^4)}
\leq& C \bigg \| |\widetilde{J}-\lambda_0|\, |\nabla^\beta \widetilde{J}| \,|\nabla^\gamma \widetilde{J}| \bigg\|_{L^{q_0}(\R^4)} \\
\leq& C \|\widetilde{J}-\lambda_0 \|_{L^{q_1}(\R^4)}
\|\nabla^\beta \widetilde{J} \|_{L^{\frac{4}{|\beta|}}(\R^4)}
\| \nabla^\gamma \widetilde{J}\|_{L^{\frac{4}{|\gamma|}}(\R^4)} \\
\leq& C \|\nabla \widetilde{J} \|_{L^{p_0}(\R^4)} E(\widetilde{J},\infty)^{N_{\beta,\gamma}} \\
\leq& C \|\nabla J \|_{L^{p_0}(B_1)} E(J,1)^{N_{\beta,\gamma}}
\end{align*}
where we let $\frac{4}{s}:=\infty$ for $s=0$, $N_{\beta,\gamma}$ stands for the number of non-zero elements in $\{\beta,\gamma\}$ and $q_0, q_1 \in (1,\infty)$ satisfy
\begin{align*}
\frac{1}{p_0}+1 =& \frac{|\alpha|+1}{4} + \frac{1}{q_0} \\
\frac{1}{q_0} =& \frac{1}{q_1} +\frac{|\beta|}{4}+ \frac{|\gamma|}{4}.
\end{align*}
Since $|\alpha|+|\beta|+|\gamma|=4$ and $1\leq|\alpha| \leq 3$, we know that $1\leq N_{\beta,\gamma}\leq 2$ and above two equations for $q_0$ and $q_1$ are solvable. Hence, if $E(J,1)\leq 1$, there holds
\begin{align}\label{lem:S:decay:temp1}
\|\nabla \omega_{\alpha,\beta, \gamma} \|_{L^{p_0}(B_1)}
\leq \|\nabla \omega_{\alpha,\beta, \gamma} \|_{L^{p_0}(\R^4)}
\leq C \|\nabla J \|_{L^{p_0}(B_1)} E(J,1).
\end{align}
By a similar argument, we also have
\begin{align}\label{lem:S:decay:temp2}
\|\nabla \omega_{\alpha,\delta} \|_{L^{p_0}(B_1)}
\leq \|\nabla \omega_{\alpha,\delta} \|_{L^{p_0}(\R^4)}
\leq C \|\nabla J \|_{L^{p_0}(B_1)} E(J,1).
\end{align}
Combining (\ref{lem:S:decay:temp1}) and (\ref{lem:S:decay:temp2}), we deduce (\ref{lem:S:decay:main}).

Finally, we turn to proving (\ref{lem:S:decay}). Let $v(x):=J(x)-\omega(x)$, then we know $v(x)$ is biharmonic on unit ball $B_1$, i.e., $\Delta^2 v(x)=0$. Since $\nabla v$ is also biharmonic, it follows from Lemma \ref{lem:m-ellpitic} (or see Lemma 6.2 in \cite{GS}) that there holds
\begin{align*}
\| \nabla v \|_{L^\infty(B_{\frac{1}{2}})}
\leq C \| \nabla v \|_{L^1(B_1)}.
\end{align*}
Hence, for any $\theta \in (0,\frac{1}{2})$ and $E(u,1)\leq 1$, there holds
\begin{align*}
D_{p_0}(J,\theta)
=& \theta^{1-\frac{4}{p_0}} \| \nabla J\|_{L^{p_0}(B_\theta)} \\
\leq& \theta^{1-\frac{4}{p_0}}
\big( \| \nabla \omega(x)\|_{L^{p_0}(B_\theta)}+ \|\nabla v(x)\|_{L^{p_0}(B_\theta)} \big)\\
\leq& C \theta^{1-\frac{4}{p_0}}
\big( \| \nabla \omega(x)\|_{L^{p_0}(B_\theta)}
+ \theta^{\frac{4}{p_0}}\|\nabla v(x)\|_{L^{\infty}(B_\theta)} \big)\\
\leq& C \theta^{1-\frac{4}{p_0}}
\big( \| \nabla \omega(x)\|_{L^{p_0}(B_\theta)}
+ \theta^{\frac{4}{p_0}}\|\nabla v(x)\|_{L^{\infty}(B_{\frac{1}{2}})} \big)\\
\leq& C \theta^{1-\frac{4}{p_0}}
\big( \| \nabla \omega(x)\|_{L^{p_0}(B_\theta)}
+ \theta^{\frac{4}{p_0}}\|\nabla v(x)\|_{L^{p_0}(B_1)} \big)\\
\leq& C \theta^{1-\frac{4}{p_0}}
\big( \| \nabla \omega(x)\|_{L^{p_0}(B_\theta)}
+ \theta^{\frac{4}{p_0}}\|\nabla \omega(x)\|_{L^{p_0}(B_1)}
+\theta^{\frac{4}{p_0}}\|\nabla J(x)\|_{L^{p_0}(B_1)} \big)\\
\leq& C \theta^{1-\frac{4}{p_0}}
\big( \| \nabla \omega(x)\|_{L^{p_0}(B_\theta)}
+\theta^{\frac{4}{p_0}}\|\nabla J(x)\|_{L^{p_0}(B_1)} \big)\\
\leq& C
\big(  \theta^{1-\frac{4}{p_0}}E(J,1) \| \nabla J(x)\|_{L^{p_0}(B_\theta)}
+\theta\|\nabla J(x)\|_{L^{p_0}(B_1)} \big)\\
\leq& C
\big(  \theta^{1-\frac{4}{p_0}} E(J,1) +\theta \big) D_{p_0}(J,1).
\end{align*}
Thus, for any give $\tau \in (0,1)$, by choosing $\theta=\theta_0$ and $\epsilon_0$ sufficiently small, we obtain (\ref{lem:S:decay}) for $E(J,1)\leq \epsilon_0$. The proof is complete.
\end{proof}

\subsection{Decay estimates for a class of semilinear elliptic equations}
It is easily seen that the property (\ref{eqn:biharm:lambda}) which biharmonic almost complex structures satisfy plays an important role in the proof of Lemma \ref{lem:decay_biharmonic}. Based on this observation, we can generalize the result in Section \ref{sec:decay_biharm} to a class of semilinear elliptic equations which admit simliar structures. To be more precise, we give the following definition.

\begin{definition}\label{def:good divergence form}
We say that the equation (\ref{eqn:semilinear}) admits a good divergence form if for any fixed constant vector $\lambda_0 \in \R^K$, $\Psi$ can be decomposed into two parts, the highest order term $\Psi_H$ and the lower order term $\Psi_L$, i.e.,
\begin{align}
\Psi=\Psi_H+\Psi_L.
\end{align}
which satisfy the following properties:
\begin{enumerate}
\item $\Psi_H $ is  a linear combination of the following terms
\begin{align}\label{def:structure}
\nabla^\alpha  ((u-\lambda_0) * h_{\alpha, \beta}),
\quad \mbox{with}\,\,\,
|h_{\alpha, \beta}| \leq C \prod_{i=1}^{s} \big|\nabla^{\beta_i} u \big|,
\end{align}
where $\alpha, \beta_i$ are multi-indices such that
\begin{align}
&|\alpha|+\sum_{i=1}^{s} |\beta_i|= 2m,  \label{def:cond:1}\\
&|\beta_i| \leq m,  i=1, \cdots , s, \quad s \in \mathbb{N}^+,  \label{def:cond:2}\\
&1 \leq \sum_{i=1}^s|\beta_i| \leq 2m-1, \label{def:cond:3}
\end{align}
and  $\beta=(\beta_1, \cdots, \beta_s)$. The notation $A*B$ means the composition of terms $A$ and $B$, such as $AB$ and $BA$.
\item $\Psi_L$ is a linear combination of the following three types of terms
\begin{equation}\label{def:structure:lower}
\begin{split}
&\nabla^{\alpha}(a_{\alpha,\gamma}(x)*\ell_{\alpha,\gamma}),
\quad \mbox{with}\,\,\,
|\ell_{\alpha,\gamma}| \leq C \prod_{i=1}^{s} \big|\nabla^{\gamma_i} u \big|, \\
&b_t(x)*\big(u(x)-\lambda_0 \big)*\ell_{0,t}),
\quad \mbox{with}\,\,\,
|\ell_{0,t}| \leq C |u|^t, \quad t \in \mathbb{N}, \\
& c(x)
\end{split}
\end{equation}
where $\gamma=(\gamma_1, \cdots, \gamma_s)$, $a_{\alpha,\gamma}(x), b_t(x), c(x) \in C^{2m}(\overline{B_1}, \R^K)$ and
\begin{align}
&|\alpha|+\sum_{i=1}^{s} |\gamma_i|\leq 2m-1,  \label{def:cond:4} \\
&|\gamma_i| \leq m,  i=1, \cdots , s, \quad s \in \mathbb{N}^+, \label{def:cond:5} \\
&\sum_{i=1}^{s} |\gamma_i| \geq 1. \label{def:cond:6}
\end{align}
\end{enumerate}
\end{definition}

\begin{remark}\label{rem:indices}
~
\begin{enumerate}
\item
The condition (\ref{def:cond:2}) and (\ref{def:cond:5}) are natural for us to define the weak solution to (\ref{eqn:semilinear}) for $u \in W^{m,2}$. Of courese, it will be more interesting for us to study the regularity of $W^{m,2}$ weak solution to (\ref{eqn:semilinear}) under the condition $n\geq 2m$. Otherwise, by Sobolev embedding theorem, $W^{m,2}$ can be embedded into $C^\alpha$ for $n<2m$.
\item
We point out that the condition (\ref{def:cond:3}) plays an important role in proving the H\"older continuousness of $u$ in critical dimension $n=2m$ under the structure (\ref{def:structure}) of $\Psi$. Moreover, the conditions (\ref{def:cond:1}) and (\ref{def:cond:3}) implies that for the highest order term $\Psi_H$ there always holds
\begin{align*}
1 \leq |\alpha|\leq 2m-1.
\end{align*}
However, $|\alpha|=0$ is valid for the lower order term $\Psi_L$.
\item
We claim that the terms in the form
\begin{align*}
g(x)*\nabla^{\alpha_1}u * \cdots * \nabla^{\alpha_t}u,
 \quad \mbox{where} \quad
 g(x) \in C^{4m}(\overline{B_1}, \R^K), \quad \sum_i |\alpha_i| \leq 2m-1
\end{align*}
can always be rewritten as a linear combination of terms (\ref{def:structure:lower}). The proof is quite straightforward. For convenience, we give the detailed proof. Without loss of generality we can assume $|\alpha_1| \geq |\alpha_2| \geq \cdots \geq |\alpha_t|$. If $|\alpha_1|\leq m$, the conclusion holds. If $|\alpha_1| > m$, then we set $\alpha_1=\alpha+\beta$ with $|\beta|=m$ and deduce that $|\alpha|< m$ and $|\alpha_i|+ |\alpha| <m$ ($i=2, \cdots, t$) due to the condition $\sum_i |\alpha_i| \leq 2m -1$. It follows that
\begin{align*}
& \, g(x)*\nabla^{\alpha_1}u * \cdots * \nabla^{\alpha_t}u \\
=& \nabla^\alpha
\bigg( g(x) \ast \nabla^\beta u \ast \nabla^{\alpha_2} u \ast \cdots \ast \nabla^{\alpha_t} u\bigg)
- \nabla^\alpha g(x) \ast \nabla^{\beta}u \ast \nabla^{\alpha_2} u \ast \cdots \ast \nabla^{\alpha_t}u \\
&- \sum_{i=2}^t g(x) \ast \nabla^\beta u \ast \nabla^{\alpha_2} u \ast \cdots \ast \nabla^{\alpha_i + \alpha} u \ast \cdots \ast \nabla^{\alpha_t} u
\end{align*}
which is the desired conclusion. Hence, the structure of the highest order term in  $\Psi$ is in some rough sense key to deciding whether the semilinear equation (\ref{eqn:semilinear}) admits a good divergence form.
\end{enumerate}
\end{remark}

\medskip

For any ball $B_r$ of radius $r$ centered at origin in $\mathbb{R}^n$, any $p>1$,  and $q_l\in (1,\infty)$ given by $\frac{1}{q_l}=\frac{1}{2}-\frac{m-l}{n}$ for $l=1,\cdots, m$ and $n\geq 2m$, denote
\begin{align}\label{eqn:normlised energy}
E(u,r)=\sum_{l=1}^m \big( r^{lq_l-n}\int_{B_r} |\nabla^l u|^{q_l} \big)^{\frac{1}{q_l}},
\end{align}
and
\begin{align}\label{eqn:Dp}
D_p(u,r)=\big( r^{p-n} \int_{B_r} |\nabla u|^p \big)^{\frac{1}{p}}.
\end{align}

\begin{lemma}\label{lem:decay estimates}
Suppose $n=2m$ and $u \in W^{m,2}(B_1, \R^K) \cap L^\infty $ satisfies (\ref{eqn:semilinear}) in distributional sense. If (\ref{eqn:semilinear}) admits a good divergence form  and $\|u\|_{L^\infty(B_1)} \leq \mathcal{B}<\infty$, then, given any $\tau \in (0,1)$, there exists $\epsilon_0>0$ and $\theta_0 \in (0, \frac{1}{2})$, which are only dependent of $\tau, \mathcal{B}, m$, such that if
\begin{align}
E(u,1)\leq \epsilon_0,
\end{align}
then we have
\begin{align}\label{lem:decay}
D_{p_0}(u,\theta_0)\leq \theta_0^{\tau} \big(D_{p_0}(u,1)+ \Lambda),
\end{align}
where $p_0=\frac{4m}{3} \in (1,2m)$ and
\begin{align}
\Lambda:=\sum_{\alpha, \gamma} \| a_{\alpha, \gamma}(x)\|_{L^\infty(B_1)}
+ \sum_t \| b_t(x) \|_{L^\infty(B_1)}+\| \nabla c(x) \|_{L^\infty(B_1)}
\end{align}
where $a_{\alpha, \gamma}(x), b_t(x), c(x)$ are from (\ref{def:structure:lower}) in lower order terms $\Psi_L$ of (\ref{eqn:semilinear}).
\end{lemma}

\begin{proof}
Firstly, it should be pointed out that due to $n=2m$, we have
\begin{align*}
E(u,r)=\sum_{l=1}^m \big( \int_{B_r} |\nabla^l u|^{\frac{2m}{l}} \big)^{\frac{l}{2m}}.
\end{align*}
For simplicity, we always denote by $C$ a positive constant only dependent of $\tau,\mathcal{B},m$ in the following proof.

\medskip

Secondly, following the similar arguemnt in the proof of Lemma \ref{lem:decay_biharmonic}, we can extend $u$ to $\widetilde{u} \in W^{m,2}(\mathbb{R}^{2m}, \R^K) \cap L^\infty$ such that
\begin{align*}
\widetilde{u}|_{B_1}=u, \quad \widetilde{u}|_{\mathbb{R}^{2m} \setminus B_2}= \lambda_0
\end{align*}
where $\lambda_0=\frac{1}{|B_1|}\int_{B_1}u$, and
\begin{align}
\| \widetilde{u}\|_{L^\infty(\mathbb{R}^{2m})}
\leq& C\, \|u \|_{L^\infty(B_1)} \label{lem:ineqy:temp1} \\
\| \nabla \widetilde{u}\|_{L^{p_0}(\mathbb{R}^{2m})}
\leq& C\, \| \nabla u\|_{L^{p_0}(B_1)} \label{lem:ineqy:temp2}\\
E(\widetilde{u},\infty) \leq& C\, E(u,1),\label{lem:ineqy:temp3}
\end{align}
where $p_0=\frac{4m}{3} \in (1, 2m)$.

Of course, by a standard extension theorem to $a_{\alpha, \gamma}(x), b_t(x) \in C^{2m}(\overline{B_1}, \R^K)$ from the lower order term $\Psi_L$, there exist the corresponding functions $\widetilde{a}_{\alpha, \gamma}(x), \widetilde{b}_t(x)\in C_0^{2m}(\R^{2m}, \R^K)$ such that
\begin{align*}
&\widetilde{a}_{\alpha, \gamma}(x)|_{B_1}=a_{\alpha, \gamma}(x),
\quad  \widetilde{b}_t (x)|_{B_1}=b_t(x), \\
&\widetilde{a}_{\alpha, \gamma}(x)|_{\R^{2m} \setminus B_2}=0,
\quad \widetilde{b}_t (x) |_{\R^{2m} \setminus B_2}=0, \\
&\| \widetilde{a}_{\alpha, \gamma}(x)\|_{L^\infty(\R^{2m})}
\leq C \|a_{\alpha, \gamma}(x) \|_{L^\infty(B_1)},\\
&\| \widetilde{b}_t(x)\|_{L^\infty(\R^{2m})}
\leq C \| b_t(x)\|_{L^\infty(B_1)}. \\
\end{align*}

Thirdly, denote $G(x)=c_m \ln |x|$ to be the fundamental solution for $\Delta^m$ on $\mathbb{R}^{2m}$, where $c_m$ is a suitable constant only dependent of $m$. Then $\nabla^{2m} G$ is a Calder\'on-Zygmund kernel. Let us define
\begin{align*}
\omega(x)
=&\sum_{\alpha, \beta} \omega_{\alpha,\beta}(x)
:=\int_{\mathbb{R}^{2m}} \nabla^{\alpha} G(x-y)
\bigg(
\big( \widetilde{u}(y)-\lambda_0 \big) * \widetilde{h}_{\alpha,\beta}(y)
\bigg) dy \\
&+\sum_{\alpha, \gamma} \omega_{\alpha,\gamma}(x)
:=\int_{\mathbb{R}^{2m}} \nabla^{\alpha} G(x-y)
\bigg(
\widetilde{a}_{\alpha, \gamma}(y) * \widetilde{\ell}_{\alpha,\gamma}(y)
\bigg) dy \\
&+\sum_{t}\omega_{0,t}(x)
:= \int_{\mathbb{R}^{2m}} G(x-y)
\bigg(\widetilde{b}_t(y)*\big(\widetilde{u}(y)-\lambda_0 \big)*\widetilde{\ell}_{0,t}(y)
\bigg) dy
\end{align*}
We claim that, for  $p_0=\frac{4m}{3} \in (1,2m)$ and $E(u,1)\leq 1$, there holds
\begin{align}\label{lem:decay:main}
\|\nabla \omega\|_{L^{p_0}(B_1)} \leq C
\bigg(
E(u,1) \| \nabla u \|_{L^{p_0}(B_1)}
+ E(u,1)\cdot \Lambda
\bigg).
\end{align}

We will prove above inequality term by term. First of all, we deal with the terms $\omega_{\alpha, \beta}$. By Lemma \ref{lem:green}, we have
\begin{align*}
\| \nabla \omega_{\alpha, \beta} \|_{L^{q_0}(\R^{2m})}
\leq & C \bigg\|\big| \widetilde{u}-\lambda_0 \big| \cdot \big|\widetilde{h}_{\alpha,\beta} \big|
\bigg\|_{L^{q_{\alpha, \beta}}(\R^{2m})}  \\
\leq & C \|\widetilde{u}-\lambda_0 \|_{L^{q_1}(\R^{2m})}
\prod_{i=1}^s \|\nabla^{\beta_i} \widetilde{u} \|_{L^{\frac{2m}{|\beta_i|}}(\R^{2m})} \nonumber \\
\leq & C \|\nabla \widetilde{u}\|_{L^{q^*_1}(\R^{2m})}
\prod_{i=1}^s \|\nabla^{\beta_i} \widetilde{u} \|_{L^{\frac{2m}{|\beta_i|}}(\R^{2m})}
\end{align*}
where $\frac{2m}{|\beta_i|}:=\infty$ for $|\beta_i|=0$, and $q_0, q_1,q_1^*, q_{\alpha, \beta} \in (1,\infty)$ satisfy  \begin{align*}
1+\frac{1}{q_0}=&\frac{|\alpha|+1}{2m}+\frac{1}{q_{\alpha, \beta}}, \\
\frac{1}{q_{\alpha, \beta}}=&\frac{1}{q_1}+\frac{1}{2m} \sum_{i=1}^s |\beta_i|,\\
\frac{1}{q_1}=&\frac{1}{q_1^*}-\frac{1}{2m}.
\end{align*}
It is easy to check that the following values solve above three equations
\begin{align*}
\frac{1}{q_0}=\frac{1}{q_1^*}=\frac{3}{4m},
\quad \frac{1}{q_1}=\frac{1}{4m},
\quad \frac{1}{q_{\alpha, \beta}}=\frac{1}{2m}(\sum_{i=1}^s |\beta_i|+\frac{1}{2}).
\end{align*}
Note that, (\ref{def:cond:3}) implies $q_{\alpha, \beta}\in (1,\infty)$. Then we have, for $p_0=\frac{4m}{3}\in (1,2m)$, there holds
\begin{align*}
\| \nabla \omega_{\alpha, \beta} \|_{L^{p_0}(B_1)}
\leq \| \nabla \omega_{\alpha, \beta} \|_{L^{p_0}(\R^{2m})}
\leq & C \|\nabla \widetilde{u}\|_{L^{p_0}(\R^{2m})}
\prod_{i=1}^s \|\nabla^{\beta_i} \widetilde{u} \|_{L^{\frac{2m}{|\beta_i|}}(\R^{2m})}\\
\leq & C \|\nabla u\|_{L^{p_0}(B_1)} E(u,1)^{n_{\beta}} \|u \|_{L^\infty(B_1)}^{s-n_\beta} \\
\leq & C \mathcal{B}^{s-n_\beta} \|\nabla u\|_{L^{p_0}(B_1)} E(u,1)^{n_{\beta}}
\end{align*}
where $n_\beta=\big|\{\beta_i: \beta_i \neq 0 \}\big|\geq 1$.
Hence, if $E(u,1)\leq 1$, there holds
\begin{align}\label{lem:decay:high}
\| \nabla \omega_{\alpha, \beta} \|_{L^{p_0}(B_1)}
\leq C E(u,1) \|\nabla u\|_{L^{p_0}(B_1)}.
\end{align}

Now we proceed analogously to deal with terms $\omega_{\alpha, \gamma}$. Similarly, by Lemma \ref{lem:green}, we obtain that
\begin{align*}
\| \nabla \omega_{\alpha, \gamma} \|_{L^{q_0}(\R^{2m})}
\leq & C \bigg\|\big| \widetilde{a}_{\alpha, \gamma} \big| \cdot \big|\widetilde{\ell}_{\alpha,\gamma} \big|
\bigg\|_{L^{q_{\alpha, \gamma}}(\R^{2m})}  \\
\leq & C \big\| \widetilde{a}_{\alpha, \gamma} \big\|_{L^{q_1}(\R^{2m})}
\prod_{i=1}^s \big\|\nabla^{\gamma_i} \widetilde{u} \big\|_{L^{\frac{2m}{|\gamma_i|}}(\R^{2m})}
\end{align*}
where $\frac{2m}{|\gamma_i|}:=\infty$ for $|\gamma_i|=0$, and $q_0, q_1, q_{\alpha, \gamma} \in (1,\infty)$ satisfy
\begin{align*}
1+\frac{1}{q_0}=& \frac{|\alpha|+1}{2m}+\frac{1}{q_{\alpha, \gamma}}, \\
\frac{1}{q_{\alpha, \gamma}}=& \frac{1}{q_1}+ \frac{1}{2m} \sum_{i=1}^s |\gamma_i|.
\end{align*}
Let us take the following values to solve above two equations
\begin{align*}
\frac{1}{q_0}=\frac{1}{4m}, \quad
\frac{1}{q_{\alpha, \gamma}}=\frac{1}{4m}+1-\frac{|\alpha|+1}{2m}, \quad
\frac{1}{q_1}=\frac{1}{4m}+1-\frac{1}{2m} \big( |\alpha|+\sum_{i=1}^s |\gamma_i|+1\big).
\end{align*}
Due to (\ref{def:cond:4}), we know all above values are reasonable and $\frac{1}{q_1}\in [\frac{1}{4m}, \frac{4m-3}{4m}]$ especially. Hence, for $p_0=\frac{4m}{3} \in (1,2m)$, there holds
\begin{align*}
\| \nabla \omega_{\alpha, \gamma} \|_{L^{p_0}(B_1)}
\leq C \| \nabla \omega_{\alpha, \gamma} \|_{L^{q_0}(B_1)}
\leq&  C \big\| \widetilde{a}_{\alpha, \gamma} \big\|_{L^{q_1}(\R^{2m})}
\prod_{i=1}^s \big\|\nabla^{\gamma_i} \widetilde{u} \big\|_{L^{\frac{2m}{|\gamma_i|}}(\R^{2m})} \\
\leq& C \cdot \Lambda \cdot E(u,1)^{n_\gamma} \|u \|_{L^\infty(B_1)}^{s-n_\gamma}
\end{align*}
where $n_\gamma=\big|\{\gamma_i: \gamma_i \neq 0 \}\big|\geq 1$ due to (\ref{def:cond:6}).
Hence, if $E(u,1)\leq 1$, there holds
\begin{align}\label{lem:decay:lower-1}
\| \nabla \omega_{\alpha, \gamma} \|_{L^{p_0}(B_1)}
\leq C \cdot \Lambda \cdot E(u,1)
\end{align}
Similar argument applies to terms $\omega_{0,t}$ and yields
\begin{align}\label{lem:decay:lower-2}
\| \nabla \omega_{0, t} \|_{L^{p_0}(B_1)}
\leq C \cdot \Lambda \cdot E(u,1).
\end{align}
Combining (\ref{lem:decay:high}), (\ref{lem:decay:lower-1}) and (\ref{lem:decay:lower-2}) gives (\ref{lem:decay:main}).

\medskip

Finally, we are in a position to prove (\ref{lem:decay}). Denote $v(x):=u(x)-\omega(x)$, then we know $v(x)$ satisfies the following equation in distributional sense,
\begin{align*}
\Delta^m v(x)=c(x) \quad \mbox{on} \,\, B_1.
\end{align*}
Then, $c(x) \in C^{2m}(\overline{B_1},\R^K)$ implies $v(x) \in C^{4m}(B_1, \R^K)$. By Lemma \ref{lem:m-ellpitic}, we have
\begin{align}
\|\nabla v(x)\|_{L^{\infty}(B_{\frac{1}{2}})} \leq C
\big( \|\nabla v(x)\|_{L^1(B_1)} +\|\nabla c(x) \|_{L^\infty(B_1)} \big),
\end{align}
Hence, for any $\theta \in (0,\frac{1}{2})$ and $E(u,1) \leq 1$, there holds
\begin{align*}
D_{p_0}(u,\theta)= &\theta^{1-\frac{2m}{p_0}} \|\nabla u \|_{L^{p_0}(B_\theta)} \\
\leq & \theta^{1-\frac{2m}{p_0}} \|\nabla v \|_{L^{p_0}(B_\theta)}
+\theta^{1-\frac{2m}{p_0}} \|\nabla \omega \|_{L^{p_0}(B_\theta)}\\
\leq & C \theta \|\nabla v \|_{L^{\infty}(B_\theta)}
+\theta^{1-\frac{2m}{p_0}} \|\nabla \omega \|_{L^{p_0}(B_1)} \\
\leq & C \theta \|\nabla v \|_{L^{\infty}(B_{\frac{1}{2}})}
+\theta^{1-\frac{2m}{p_0}} \|\nabla \omega \|_{L^{p_0}(B_1)} \\
\leq & C \theta
\big( \|\nabla v \|_{L^{p_0}(B_1)}+\|\nabla c(x) \|_{L^\infty(B_1)} \big)
+\theta^{1-\frac{2m}{p_0}} \|\nabla \omega \|_{L^{p_0}(B_1)} \\
\leq & C \theta
\big( \|\nabla u \|_{L^{p_0}(B_1)}+\|\nabla \omega \|_{L^{p_0}(B_1)}+ \Lambda \big)
+\theta^{1-\frac{2m}{p_0}} \|\nabla \omega \|_{L^{p_0}(B_1)} \\
\leq & C \bigg(
\theta \big(\|\nabla u \|_{L^{p_0}(B_1)}+ \Lambda \big)
+\theta^{1-\frac{2m}{p_0}} \|\nabla \omega \|_{L^{p_0}(B_1)}
\bigg)\\
\leq & C \bigg(
\theta \big(\|\nabla u \|_{L^{p_0}(B_1)}+ \Lambda \big)
+\theta^{1-\frac{2m}{p_0}} E(u,1)
\big(\|\nabla u \|_{L^{p_0}(B_1)} + \Lambda \big)
\bigg)\\
\leq & C
 \bigg(\theta+ \theta^{1-\frac{2m}{p_0}} E(u,1)\bigg)
\big(\|\nabla u \|_{L^{p_0}(B_1)} + \Lambda \big) \\
\leq & C
 \bigg(\theta+ \theta^{1-\frac{2m}{p_0}} E(u,1)\bigg)
\big(D_{p_0}(u,1) + \Lambda \big).
\end{align*}
Thus, for any given $\tau \in (0,1)$, by choosing $\theta=\theta_0$ and $\epsilon_0$ sufficiently small, we obtain (\ref{lem:decay}) for $E(u,1)<\epsilon_0$. The proof is complete.
\end{proof}

\subsection{Higher regularity for a class of semilinear elliptic equations}
Now we turn to generalizing the higher regularity results of a class of semilinear elliptic equations in \cite{GS}. Since the proof of the following theorem is similar to that of Proposition 7.1 in \cite{GS}, we only give the modifications that is essential to the proof.

\begin{theorem}\label{thm:higher regularity}
Suppose $n\geq 2m$ and $u\in W^{m,2}(B_1,\R^K) \cap C^{0,\mu}$ satisfies (\ref{eqn:semilinear}) in distributional sense, where $\Psi$ can be divided into two parts: the highest order terms $H$ and lower order terms $L$, i.e., $\Psi=H+L$, which admit the following structures:
\begin{align}\label{thm:HR:eqn:H}
H=\sum_{k=0}^{m-1} \nabla^k \cdot g_k,
\quad \mbox{where} \quad
|g_k| \leq C \sum_{l=1}^m |\nabla^l u|^{\frac{2m-k}{l}},
\end{align}
and
\begin{align}\label{thm:HR:eqn:L}
L=\sum_{k=0}^{m-1} \nabla^k \cdot \widetilde{g}_k,
\quad \mbox{where} \,\,
|\widetilde{g}_k| \leq C \sum_{\gamma} \bigg( \prod_{i} |\nabla^{\gamma_i}u| \bigg)
\,\, \mbox{with} \,\,
\sum_{i}|\gamma_i| \leq 2m-1-k.
\end{align}
Then, $u\in C^\infty(B_1, \R^K)$.
\end{theorem}

\begin{proof}
For $m=1$, (\ref{eqn:semilinear}) is just the second order semilinear equation and the conclusion holds obviously. Thus, we focus on the case $m \geq 2$ in the following proof.

\medskip
It is clear that Gastel and Scheven in \cite{GS} proved the theorem in the case $\Psi=H$. According to the proof of Proposition 7.1 in \cite{GS}, it suffices to prove the following two claims in the case $\Psi=L$:
\begin{enumerate}
\item [(1)]
\begin{align} \label{lem:smooth:claim-1}
\sup_{B_\rho(x)\subset B_R} \rho^{2m-n-2\mu}
\int_{B_\rho(x)}|\nabla^m u|^2  <\infty, \quad \forall\, \,0<R<1,
\end{align}
\item [(2)] For every non-integer $\nu:=[\nu]+\sigma \in (0,m)$, if $u\in C^{[\nu],\sigma}(B_1, \R^K)$ and
\begin{align}\label{lem:smooth:claim-2-cond}
\sup_{B_\rho(x)\subset B_R} \rho^{2m-n-2\nu}
\int_{B_\rho(x)}|\nabla^m u|^2  <\infty, \quad \forall \,0<R<1,
\end{align}
then we have that, for $0\leq k\leq m-1$ and $B_\rho(x) \subset B_R$, there holds
\begin{align}\label{lem:smooth:claim-2}
\bigg( \rho^{2m-n}\int_{B_\rho(x)} |\widetilde{g}_k|^{\frac{2m}{2m-k}} \bigg)^{\frac{2m-k}{2m}}
\leq C \rho^{\frac{m+1}{m}\nu}
\end{align}
\end{enumerate}

Before proceeding to prove claims, we make some conventions: fix $R\in (0,1)$,  always assume $B_\rho(x) \subset B_R$, and $C$ stand for the positive constants only dependent of $m,n,\|u\|_{C^{0,\mu}(B_R)}$.

\medskip
We first prove the Claim (1) by standard integral estimates. Since $u\in C^{0,\mu}(B_1)$, we have
\begin{align*}
\| u-\overline{u}\|_{L^\infty(B_\rho(x))} \leq C [u]_{\mu;B_R} \rho^\mu \leq C \rho^\mu,
\end{align*}
where $\overline{u}=\frac{1}{|B_\rho(x)|}\int_{B_\rho(x)} u(y)$. To simplify the proof in the following estimate, we assume $\| u-\overline{u}\|_{L^\infty(B_\rho(x))} \leq 1$.

By Gagliardo–Nirenberg interpolation inequality, we have that, for $1\leq l\leq m-1$, there holds
\begin{align}\label{lem:smooth:ineqy:Nirenberg1}
\begin{split}
\bigg( \rho^{2m-n} \int_{B_\rho(x)}|\nabla^l u|^{\frac{2m}{l}} \bigg)^{\frac{l}{2m}}
\leq & C \| u-\overline{u}\|_{L^\infty}^{1-\frac{l}{m}}
\bigg( \rho^{2m-n} \int_{B_\rho(x)}|\nabla^m u|^2 \bigg)^{\frac{l}{2m}} \\
&\, +C \| u-\overline{u}\|_{L^\infty}.
\end{split}
\end{align}
It follows that
\begin{align}\label{lem:smooth:ineqy:Nirenberg2}
\int_{B_\rho(x)} |\nabla^l u|^{\frac{2m}{l}} \leq& C
\| u-\overline{u}\|_{L^\infty}^{(1-\frac{l}{m})\frac{2m}{l}}
\int_{B_\rho(x)}|\nabla^m u|^2
+ C \rho^{n-2m}\| u-\overline{u}\|_{L^\infty}^{\frac{2m}{l}} \nonumber \\
\leq& C \| u-\overline{u}\|_{L^\infty}^{\frac{2}{m}}
\int_{B_\rho(x)}|\nabla^m u|^2
+ C \rho^{n-2m}\| u-\overline{u}\|_{L^\infty}^{2}.
\end{align}
On the other hand, by H\"older's inequality, it follows from (\ref{lem:smooth:ineqy:Nirenberg1}) that, for $1\leq l \leq m-1$ and $q \in [1,\frac{2m}{l}]$
\begin{align}\label{lem:smooth:ineqy:Lq}
\begin{split}
\bigg( \rho^{lq-n} \int_{B_\rho(x)}|\nabla^l u|^{q} \bigg)^{\frac{1}{q}}
\leq & C \| u-\overline{u}\|_{L^\infty}^{1-\frac{l}{m}}
\bigg( \rho^{2m-n} \int_{B_\rho(x)}|\nabla^m u|^2 \bigg)^{\frac{l}{2m}} \\
&\, +C \| u-\overline{u}\|_{L^\infty}.
\end{split}
\end{align}

\medskip

We choose a cut-off function $\eta\in C^\infty_0(B_\rho(x),[0,1])$ such that
\begin{align*}
\eta|_{B_{\frac{\rho}{2}}(x)} \equiv 1 \quad \mbox{and} \quad
\|\nabla^l \eta\|_{L^\infty} \leq C \rho^{-l}, \quad \forall \, l\in \mathbb{N}.
\end{align*}

Testing (\ref{eqn:semilinear}) with $\eta^{2m}(u-\overline{u})$, we compute
\begin{align}
\int \eta^{2m} |\nabla^m u|^2 dy
\leq& C \sum_{k=0}^{m-1} \int |\nabla^m u| \cdot |\nabla^k(u-\overline{u})| \cdot |\nabla^{m-k} \eta^{2m}| \nonumber \\
&+ C \sum_{k=0}^{m-1} \sum_{j=0}^{k} \int |\nabla^j (u-\overline{u})| \cdot |\nabla^{k-j} \eta^{2m}| \cdot |\widetilde{g}_k| \nonumber \\
&=: \sum_{k=0}^{m-1} I_k + \sum_{k=0}^{m-1} \sum_{j=0}^{k} II_{kj}. \label{lem:smooth:temp-0}
\end{align}
We will deal with above inequality term by term. Let us estimate $I_0$ :
\begin{align}
I_0 \leq& C \int |\nabla^m u| \cdot |u-\overline{u}| \cdot |\nabla^{m} \eta^{2m}| \nonumber \\
\leq& C \rho^{-m} \|u-\overline{u}\|_{L^\infty} \int |\nabla^m u| \eta^m \nonumber\\
\leq& C \rho^{\frac{n}{2}-m} \|u-\overline{u}\|_{L^\infty}
\bigg(\int \eta^{2m} |\nabla^m u|^2\bigg)^{\frac{1}{2}}  \nonumber \\
\leq& \epsilon_1 \int \eta^{2m} |\nabla^m u|^2
+ C_{\epsilon_1} \rho^{n-2m}\|u-\overline{u}\|_{L^\infty}^2 \nonumber \\
\leq& \epsilon_1 \int \eta^{2m} |\nabla^m u|^2
+ C_{\epsilon_1} \rho^{n-2m+2\mu}
\label{lem:smooth:temp-I0}
\end{align}
where $\epsilon_1>0$ will be determined later. For $1\leq k \leq m-1$, we obtain
\begin{align}
I_k \leq& C \rho^{k-m} \int |\nabla^m u| \cdot |\nabla^k u| \cdot \eta^{m+k} \nonumber \\
\leq& C \rho^{k-m}
\bigg(\int \eta^{2m} |\nabla^m u|^2\bigg)^{\frac{1}{2}}
\bigg(\int \eta^{2k} |\nabla^k u|^2\bigg)^{\frac{1}{2}}
\nonumber \\
\leq& \epsilon_1 \int \eta^{2m} |\nabla^m u|^2
+ C_{\epsilon_1} \rho^{2k-2m} \int \eta^{2k} |\nabla^k u|^2 \nonumber \\
\leq& \epsilon_1 \int \eta^{2m} |\nabla^m u|^2
+ C_{\epsilon_1} \rho^{2k-2m}
\bigg( \epsilon_2 \rho^{2m-2k} \int_{B_\rho(x)}|\nabla^m u|^2
+ C_{\epsilon_2} \rho^{n-2k}\| u-\overline{u}\|_{L^\infty}^2\bigg) \nonumber \\
\leq& \epsilon_1 \int \eta^{2m} |\nabla^m u|^2
+ \epsilon_2 C_{\epsilon_1} \int_{B_\rho(x)}|\nabla^m u|^2
+ C_{\epsilon_1}C_{\epsilon_2} \rho^{n-2m}\| u-\overline{u}\|_{L^\infty}^2 \nonumber \\
\leq& C\epsilon_1 \int \eta^{2m} |\nabla^m u|^2
+ \epsilon_2 C_{\epsilon_1} \int_{B_\rho(x)}|\nabla^m u|^2
+ C_{\epsilon_1}C_{\epsilon_2} \rho^{n-2m+2\mu}
\label{lem:smooth:temp-Ik}
\end{align}
where we use (\ref{lem:smooth:ineqy:Lq}) with $q=2$ and Young's inequality in the fourth inequality, and $\epsilon_2>0$ will be determined later.
Next, we estimate $II_{00}$ as follows
\begin{align}
II_{00}
\leq& C \sum_\gamma \int |u-u| \cdot \eta^{2m} \cdot \prod_{i}|\nabla^{\gamma_i}u| \nonumber \\
\leq&  C \sum_\gamma \rho^{\frac{n}{p_0}}\cdot \| u-\overline{u}\|_{L^\infty}
\prod_i \| \eta^{|\gamma_i|}\nabla^{\gamma_i} u\|_{L^{\frac{2m}{|\gamma_i|}}} \nonumber \\
\leq& C \| u-\overline{u}\|_{L^\infty}
\bigg(\rho^n + \sum_{\gamma_i \neq 0} \int \eta^{2m} |\nabla^{|\gamma_i|} u|^{\frac{2m}{|\gamma_i|}}\bigg) \nonumber \\
\leq& C \| u-\overline{u}\|_{L^\infty}
\bigg(\rho^n + \sum_{l=1}^m \int_{B_{\rho}(x)} |\nabla^l u|^{\frac{2m}{l}}\bigg) \nonumber \\
\leq& C \| u-\overline{u}\|_{L^\infty} \int_{B_\rho(x)} |\nabla^m u|^2
+ C \rho^n \| u-\overline{u}\|_{L^\infty}  \nonumber \\
& + C
\bigg(
\| u-\overline{u}\|_{L^\infty}^{1+\frac{2}{m}}
\int_{B_\rho(x)} |\nabla^m u|^2
+\rho^{n-2m} \| u-\overline{u}\|_{L^\infty}^{3}
\bigg) \nonumber \\
\leq& C \rho^{\mu} \int_{B_\rho(x)} |\nabla^m u|^2 + C \rho^{n-2m+2\mu}
\label{lem:smooth:temp-II0}
\end{align}
where we use (\ref{lem:smooth:ineqy:Nirenberg2}) in the fifth inequality, and
\begin{align*}
1=\frac{1}{p_0}+\frac{1}{2m}\sum_i |\gamma_i|.
\end{align*}
Note that, due to $\sum_i |\gamma_i| \leq 2m-1$, it follows that $p_0 \in [1, 2m]$.
Similar arguments apply to $II_{k,0}$ and we obtain, for $1\leq k \leq m-1$,
\begin{align}
II_{k0} \leq C \rho^{\mu} \int_{B_\rho(x)} |\nabla^m u|^2 + C \rho^{n-2m+2\mu}.
\end{align}
By H\"older's inequality and Young's inequality, we have that, for $1\leq k \leq m-1$,
\begin{align}\label{lem:smooth:temp-II}
\int \eta^{2m} |\widetilde{g}_k|^{\frac{2m}{2m-k}}
\leq& \epsilon_2 \int_{B_\rho(x)} |\nabla^m u|^2 + C_{\epsilon_2}
\bigg( \rho^n + \sum_{l=1}^{m-1} \int_{B_\rho(x)} |\nabla^l u|^{\frac{2m}{l}} \bigg) \nonumber \\
\leq& (\epsilon_2+C_{\epsilon_2} \rho^{\frac{2\mu}{m}}) \int_{B_\rho(x)} |\nabla^m u|^2 + C_{\epsilon_2} \rho^{n-2m+2\mu},
\end{align}
where we apply (\ref{lem:smooth:ineqy:Nirenberg2}) in second inequality. Now let us turn to estimating $II_{kj}$ with $1\leq j <k\leq m-1$,
\begin{align} \label{lem:smooth:temp-IIkj}
II_{kj} \leq& C \rho^{j-k} \int |\widetilde{g}_k| \cdot |\nabla^j u| \cdot \eta^{2m+j-k} \nonumber \\
\leq& C \rho^{j-k} \| \eta^j \nabla^j u\|_{L^{\frac{2m}{k}}}
\| \eta^{2m-k} \widetilde{g}_k\|_{L^{\frac{2m}{2m-k}}} \nonumber \\
\leq& C  \rho^{(j-k)\frac{2m}{k}} \int_{B_\rho(x)} |\nabla^j u|^{\frac{2m}{k}}
+ C \int \eta^{2m} |\widetilde{g}_k|^{\frac{2m}{2m-k}} \nonumber \\
\leq & C (\epsilon_2+C_{\epsilon_2}\rho^{\frac{2\mu}{m}}) \int_{B_\rho(x)} |\nabla^m u|^2
+ C_{\epsilon_2} \rho^{n-2m+2\mu}
\end{align}
where we use (\ref{lem:smooth:ineqy:Lq}) with $q=\frac{2m}{k}$ and (\ref{lem:smooth:temp-II}) in the last inequality. Similarly, we obtain, for $1\leq k \leq m-1$
\begin{align} \label{lem:smooth:temp-IIkk}
II_{kk}\leq (\epsilon_2 + C_{\epsilon_2} \rho^{\frac{2\mu}{m}})
\int_{B_\rho(x)} |\nabla^m u|^2
+ C_{\epsilon_1} \rho^{n-2m+2\mu}.
\end{align}

Combining above all estimates, we deduce that
\begin{align*}
\int \eta^{2m} |\nabla^m u|^2 dy
\leq&  C\epsilon_1 \int \eta^{2m} |\nabla^m u|^2
+ C (\epsilon_2+C_{\epsilon_2}\rho^{\frac{2\mu}{m}}) \int_{B_\rho(x)} |\nabla^m u|^2
+ C_{\epsilon_1,\epsilon_2} \rho^{n-2m+2\mu}.
\end{align*}
Thus, by choosing $\epsilon_1, \epsilon_2, \rho_0$ small enough, we have that, for all $\rho\leq \rho_0$, there holds
\begin{align*}
\int_{B_{\frac{\rho}{2}}} |\nabla^m u|^2 \leq \varepsilon
\int_{B_{\rho}} |\nabla^m u|^2 + C \rho^{n-2m+2\mu}
\end{align*}
where $\varepsilon < 2^{2m-n-2\mu}$ is a fixed positive number. A standard iteration argument implies (\ref{lem:smooth:claim-1}).

\medskip
The task is now to prove Claim (2). Since $u \in C^{[\nu],\sigma}(B_1)$ with $\nu=[\nu]+\sigma$, we know that, there exists a Taylor polynomials $P_x$ at the points $x$ such that
\begin{align*}
\| u \|_{C^{[\nu],\sigma}(B_R)}\leq C <\infty,\quad
\|u-P_x \|_{L^\infty(B_\rho(x))}\leq C \rho^{\nu}.
\end{align*}
By Gagliardo–Nirenberg interpolation inequality and (\ref{lem:smooth:claim-2-cond}), we have that, for $\nu< l \leq m$, there holds
\begin{align*}
\bigg( \rho^{2m-n} \int_{B_\rho(x)} |\nabla^l u|^{\frac{2m}{l}} \bigg)^{\frac{l}{2m}}
\leq& C \|u-P_x \|_{L^\infty(B_\rho(x))}^{1-\frac{l}{m}}
\bigg( \rho^{2m-n} \int_{B_\rho(x)} |\nabla^m u|^{2} \bigg)^{\frac{l}{2m}}  \\
& \,\,+C \|u-P_x \|_{L^\infty(B_\rho(x))} \\
\leq & C \rho^\nu.
\end{align*}
Let us compute
\begin{align}
\bigg( \rho^{2m-n}\int_{B_\rho(x)} |\widetilde{g}_k|^{\frac{2m}{2m-k}} \bigg)^{\frac{2m-k}{2m}}
\leq& C \rho^{(2m-n) \frac{2m-k}{2m}}\cdot \rho^{\frac{n}{q_0}}
\prod_{|\gamma_i|>\nu} \big\| \nabla^{|\gamma_i|} u \big\|_{L^{\frac{2m}{|\gamma_i|}}(B_{\rho}(x))} \nonumber \\
\leq & C \rho^{\tau}
\end{align}
where
\begin{align}\label{lem:smooth:calim2-temp}
 \frac{1}{q_0}+ \frac{1}{2m}\sum_{|\gamma_i|>\nu} |\gamma_i|=\frac{2m-k}{2m}
\end{align}
and
\begin{align*}
\tau=(2m-n) \cdot \frac{2m-k}{2m} +  \frac{n}{q_0}
+ \sum_{|\gamma_i|>\nu} \big(\nu +\frac{n-2m}{2m}|\gamma_i| \big).
\end{align*}
Combining above two identities yields
\begin{align*}
\tau=&(2m-n) \cdot \frac{2m-k}{2m}+  \frac{n}{q_0} + n_{\nu} \nu
+(n-2m) \bigg( \frac{2m-k}{2m}-\frac{1}{q_0} \bigg) \\
=&\frac{2m}{q_0}+n_\nu \nu.
\end{align*}
where $n_{\nu}=\big|\{\gamma_i: |\gamma_i|>\nu \}\big|$.

We claim that
\begin{align}\label{lem:smooth:kappa}
\tau\geq \frac{m+1}{m} \nu.
\end{align}
which implies (\ref{lem:smooth:claim-2}).
Obviously, (\ref{lem:smooth:kappa}) holds for $n_\nu \geq 2$.
For $n_\nu =0$, (\ref{lem:smooth:calim2-temp}) implies $\frac{1}{q_0}=\frac{2m-k}{2m}$. Hence, for $\nu \in (0,m)$
\begin{align*}
\tau=\frac{2m}{q_0}=2m-k\geq m+1\geq \frac{m+1}{m} \nu.
\end{align*}
For $n_\nu=1$, (\ref{lem:smooth:calim2-temp}) and the fact $k+\sum_i |\gamma_i| \leq 2m-1$ imply $\frac{1}{q_0} \geq \frac{1}{2m}$. Hence, for $\nu \in (0,m)$,
\begin{align*}
\tau=\frac{2m}{q_0}+ \nu \geq 1+ \nu \geq \frac{m+1}{m} \nu.
\end{align*}
Thus, the claim (\ref{lem:smooth:kappa}) is proved.
\end{proof}


\section{H\"older regularity for \texorpdfstring{$W^{m,2}$}{Wm,2} \texorpdfstring{$m$}{m}-harmonic almost complex structure in critical dimension \texorpdfstring{$n=2m$}{n=2m}}\label{sec:Holder}

In this section, we will establish the H\"older regularity of $W^{m,2}$ $m$-harmonic almost complex structure in critical dimension \texorpdfstring{$n=2m$}{n=2m} by employing the nonlinear structure of the corresponding equation. To be precise, we will prove the following theorem.
\begin{theorem}\label{thm:Holder}
Suppose $m \in \{2,3 \}$ and $J\in W^{m,2}(\mathcal{J}_{g})$ is a weakly $m$-harmonic almost complex structure on $(M^n,g)$ of $n=2m$. Then $J$ is H\"older-continuous.
\end{theorem}

Since the H\"older regularity is a local property in nature, we may assume $(M^n,g)$ to be $(B_1,g)$ where $B_1$ is the unit ball of $\R^n$ centered at origin and $g$ is a smooth metric on $B_1$. To illustrate the main point of the argument, we firstly consider the case $g=g_0=\sum_i dx^i \otimes dx^i$, i.e., $(B_1,g_0)$ admits the Euclidean metric.
Then, we will deal with the general case as a small perturbation of the Euclidean case.

\subsection{The Euclidean case \texorpdfstring{$(B_1, g_0)$}{(B1,g0)}}
Since the metric is Euclidean, the covariant derivatives are just ordinary derivatives, and so we can interchange the order of derivatives. Moreover, an almost complex structure $J \in W^{m,2}(\mathcal{J}_{g})$ on $B_1$ can be regarded as a function in $W^{m,2}(B_1, M_n(\mathbb{R}))$ such that $J^2=-id$ and $J^t+J=0$, where $M_n(\mathbb{R})$ is the set of all real $n\times n$ matrices and $J^t$ is the transpose of matrix $J$. By the definition of the inner product of $A, B \in T^1_1(B_1)$, we know
\begin{align}
\langle A, B\rangle =\sum_{i,j=1}^n A^j_i B^j_i
\end{align}
for $A=A_i^j dx^i \otimes \frac{\partial}{\partial x^j}$ and
$B=B_i^j dx^i \otimes \frac{\partial}{\partial x^j}$. Thus, the inner product of $(1,1)$ tensor fields on $B_1$ can be viewed as the inner product of two vectors in Euclidean space $\R^{n^2}$. Finally, there holds
\begin{align}
|J|^2=\langle J, J\rangle= \sum_{i,j=1}^n \big( J^j_i \big)^2
=\mbox{Trace} (JJ^t)=\mbox{Trace} \big(J(-J)\big)=n
\end{align}
for all almost complex structure $J$ defined on $(B_1,g_0)$, i.e., $\| J \|_{L^\infty}=\sqrt{n}$.

In order to prove the regularity of weakly $m$-harmonic almost complex structures, we need to rewrite the corresponding Euler-Lagrange equation in a good divergence form. More precisely, we have the following lemma.
\begin{lemma}\label{lem:Nonliearity}
Suppose $m\in \{2,3 \}$ and $J$ is a $W^{m,2}$ weakly $m$-harmonic almost structure on $(B_1,g_0)$. Then $J$ satisfies the following equation in distributional sense,
\begin{align}\label{eqn:m-harmonic}
\Delta^m J = \Psi(J, \nabla J, \cdots, \nabla^{2m-1}J)
\end{align}
where $\Psi$ admits the property that for any fixed constant matrix $\lambda_0\in M_n(\R)$, $\Psi$ can be rewritten as a linear combination of the following terms
\begin{align*}
\nabla^{\alpha}\ast\bigg((J-\lambda_{0})\ast\nabla^{\beta}J\ast\nabla^{\gamma}J\bigg)
\quad or\quad
\lambda_0\ast\nabla^{\alpha}\ast\bigg((J-\lambda_{0})\ast\nabla^{\delta}J\bigg),
\end{align*}
where $\alpha,\beta,\gamma,\delta$ are multi-indices such that $1\leq|\alpha| \leq 2m-1$, $0\leq |\beta|,|\gamma|,|\delta|\leq m$, $|\alpha|+|\beta|+|\gamma|=2m$
and $|\alpha|+|\delta|=2m$.
\end{lemma}
\begin{proof}
The Lemma is a direct consequence of Proposition \ref{prop:T_lambda}.
\end{proof}

It is obvious that (\ref{eqn:m-harmonic}) admits a good divergence form with $\Psi_L=0$ by Definition \ref{def:good divergence form}.

\medskip
\noindent \textbf{Proof of Theorem \ref{thm:Holder} in the Euclidean case $(B_1,g_0)$:}
\medskip

First, we use the normalized energy $E(J;x,r)$ defined by replacing $u$,$B_r$ by $J$, $B_r(x)$ respectively in (\ref{eqn:normlised energy}).
Since $n=2m$, it follows that for any $J \in W^{m,2}(B_1)$ and $B_r(x) \subset B_1$,
\begin{align*}
E(J;x,r)= \sum_{l=1}^{m}
\bigg(\int_{B_r(x)} |\nabla^l J|^{\frac{2m}{l}} \bigg)^{\frac{l}{2m}}
\end{align*}
is well defined due to Sobolev embedding theorem. For any fixed $R_0 \in (0,1)$, we have that for every $\epsilon_0>0$, there exists $r_0 \in (0, 1-R_0)$ such that
\begin{align}\label{thm:smooth:eqn:temp-1}
\sup_{x \in \overline{B}_{R_0}}E(J;x,r_0) < \epsilon_0.
\end{align}
It is easy to check that for any fixed point $x_0 \in \overline{B}_{R_0}$, $J_{x_0, r_0}(x):=J(x_0+r_0x)$ is also a $W^{m,2}$ weakly $m$-harmonic almost structure on $(B_1,g_0)$ with
\begin{align}
E(J_{x_0,r_0}; 0,1)=E(J; x_0,r_0)<\epsilon_0
\end{align}

By Lemma \ref{lem:Nonliearity}, the equation $J_{x_0,r_0}$ satisfies admits a good divergence form (see definiton \ref{def:good divergence form}) with $\Psi_L=0$.
Then it follows from Lemma \ref{lem:decay estimates} that by choosing suitable $\epsilon_0>0$ in (\ref{thm:smooth:eqn:temp-1}), there exists $\theta_0 \in (0,\frac{1}{2})$ such that
\begin{align}
D_{p_0}(J;x_0,\theta_0 r_0)=D_{p_0}(J_{x_0,r_0};0, \theta_0)
\leq \sqrt{\theta_0} D_{p_0}(J_{x_0,r_0};0,1)
=\sqrt{\theta_0} D_{p_0}(J;x_0,r_0)
\end{align}
where $p_0=\frac{4m}{3}$ and
\begin{align*}
D_{p_0}(J;x, r)=\bigg(r^{p_0-2m} \int_{B_r(x)} |\nabla J|^{p_0} \bigg)^\frac{1}{p_0}.
\end{align*}
Finally, by a standard iteration argument, there exists $\alpha \in (0,1)$ such that
\begin{align*}
D_{p_0}(J;x_0, r) \leq C r^\alpha, \quad \forall r \in (0,r_0),
\end{align*}
where $C$ is only dependent of $r_0$ and $\theta_0$. The Morrey's lemma implies that $J \in C^{0,\alpha}(\overline{B}_{R_0})$, hence that $J \in C^{0,\alpha}(B_1)$ which is the desired conclusion.
\qed

\subsection{The general case \texorpdfstring{$(B_1, g)$}{(B1,g)}}
In this subsection, we will prove the H\"older regularity of weakly $m$-harmonic almost structure on $(B_1,g)$ by a perturbation method.

\medskip
We start by recalling the scaling invariance of the functional $\mathcal{E}_m(J)$ in critical dimension $n=2m$. That is, if we rewrite the functional $\E_{m}$ in the following form to emphasize the metric $g$
\begin{align*}
\mathcal{E}_{m}(J, g)=\int_M \big| \Delta^{\frac{m}{2}}_{g} J\big|^2 dV_{g},
\end{align*}
and do the scaling $g_{\lambda}:=\lambda^2 g$ for some positive real number $\lambda$,
then
\begin{align*}
\mathcal{E}_m(J,g)= \mathcal{E}_{m}(J, g_{\lambda}).
\end{align*}
It follows that if $J$ is a weakly $m$-harmonic almost complex structure on $(M,g)$, then $J$ is also $m$-harmonic on $(M,g_{\lambda})$. Thus, doing the scaling does not affect the H\"older regularity of weakly $m$-harmonic almost complex structures in critical dimension. If we take the geodesic normal coordinates on the unit geodesic ball centered at fixed point in $(M,g_{\lambda})$, then the metric $g_{\lambda}$ in such local coordinates converges to the Euclidean metric in $C^{\infty}(B_1)$ as $\lambda$ goes to infinity. Hence, we can assume that, by a scaling if necessary, the metric $g$ on $B_1$ is sufficiently close to the Euclidean metric in the sense
\begin{align}\label{assumption:g}
|g_{ij}(x)-\delta_{ij}| + \sum_{k=1}^{2m} | D^k g_{ij}(x)| \leq \delta_0, \quad \forall x\in B_1
\end{align}
where $\delta_0$ is sufficiently small and will be determined later.

\medskip
\noindent \textbf{Proof of Theorem \ref{thm:Holder} in the general case $(B_1,g)$:}
\medskip

Firstly, we introduce an operator $\mathfrak{m}$ which maps a $(1,1)$ tensor field $A$ on $(B_1,g)$ to a $n\times n$ real matrix valued function,
\begin{align*}
A_{\m}:=\m(A)=(A_i^j)
\end{align*}
where $A=A_i^j dx^i \otimes \frac{\partial}{\partial x^j}$. In other words, $A$ denotes tensor field and $A_{\m}$ denotes its coefficient matrix.
Let us denote by $\nabla$ the covariant derivative on $(B_1,g)$ and $D$ the ordinary derivatives (i.e., $D_k=\partial_k$). Here it is necessary to emphasize the difference between the derivatives on tensor fields and matrix valued functions. For example, for $A=A_i^j dx^i \otimes \partial_j$, we have
\begin{align*}
(\nabla_{\partial_k} A )_i^j=D_k A_i^j + A_i^s \Gamma_{k s}^j-A_s^j \Gamma_{ki}^s
\end{align*}
where $\Gamma_{ij}^k$ denote the Christoffel symbols with respect to metric $g$. To simplify notation, we rewrite above equation as
\begin{align*}
\big( \nabla_{\partial_k} A \big)_{\m}=D_k A_{\m}+ Dg \ast A
\end{align*}
where $D_k A_{\m}=D_k (A_i^j)=(D_k A_i^j)$. Similarly, there holds
\begin{align}\label{eqn:holder:laplace}
\big( \Delta A \big)_{\m}= \Delta A_{\m} + Dg \ast D A_{\m}
+ \big( D^2g + Dg \ast Dg) \ast A_{\m}.
\end{align}

Now let us recall the $m$-harmonic almost complex structure equation (\ref{eqn:EL:m-poly}), i.e.,
\begin{align*}
\Delta^m J = T(J, \nabla J, \cdots, \nabla^{2m-1}J).
\end{align*}
We will reduce above equation to a perturbation form of the Euclidean case step by step. For the convenience of reader, we will give the detailed process of dealing with the term $\Delta^m J$ as a example. Repeated application of (\ref{eqn:holder:laplace}) yields
\begin{align*}
(\Delta^m J )_{\m}= \Delta^m J_{\m} + L_1(D^ig, D^j J_{\m})
\end{align*}
where $L_1$ stands for the lower order terms in the following form
\begin{align*}
L_1=\sum D^{i_1}g \ast \cdots \ast D^{i_s} g\ast D^j J_{\m}
\end{align*}
with $i_\mu \geq 1, \mu=1, \cdots, s$, $j \geq 0$ and $j+\sum_{\mu=1}^s i_{\mu}=2m$. Let us denote by $\Delta_0= \sum_{i=1}^{2m} \partial^2_i$ the standard Laplace operator on Euclidean space $\R^{2m}$. Notice that for any smooth function $f$, we know
\begin{align*}
\Delta f =g^{ij} \frac{\partial^2 f}{\partial x^i \partial x^j}- g^{ij}\Gamma^s_{ij}\frac{\partial f}{\partial x^s}
=:g^{ij} \frac{\partial^2 f}{\partial x^i \partial x^j} + Dg \ast Df,
\end{align*}
where we omit the terms $g^{ij}$ in the expression $Dg\ast Df$ due to boundedness of $g$. Then we have
\begin{align*}
\big( \Delta^m -\Delta_0^m \big)J_{\m} = P_1 + L_2 (D^i g, D^j J_{\m}),
\end{align*}
where $P_1$ stands for the perturbation term in the following form
\begin{align}\label{eqn:Holder:p1}
P_1 = \big( g^{i_1 j_1}\cdots g^{i_m j_m} - \delta^{i_1 j_1}\cdots \delta^{i_m j_m} \big) D^{2m}_{i_1 j_1 \cdots i_m j_m} J_{\m}
=:\sum_{|\beta|=2m} a_{\beta}(x) \ast D^{\beta} J_{\m},
\end{align}
and $L_2$ also stands for the lower order terms and has the similar expression as $L_1$. Hence, we obtain
\begin{align}
(\Delta^m J )_{\m}= \Delta^m_0 J_{\m} + P_1+ \widetilde{L}_1(D^ig, D^j J_{\m}),
\end{align}
where $\widetilde{L}_1=L_1+L_2$ has the following form
\begin{align}
\widetilde{L}_1=\sum D^{i_1}g \ast \cdots \ast D^{i_s} g\ast D^j J_{\m}
\end{align}
with $i_\mu \geq 1, \mu=1, \cdots, s$, $j \geq 0$ and $(\sum_{\mu=1}^s i_{\mu}) +j=2m$.

Similar arguments apply to the nonlinear term $T$ in (\ref{eqn:EL:m-poly}) and yield
\begin{align}
T(J, \nabla J, \cdots, \nabla^{2m-1} J)= T_s(J_{\m}, D J_{\m}, \cdots, D^{2m-1} J_{\m}) + P_2 + \widetilde{L}_2
\end{align}
where $T_s$ admits a good divergence form as $\Psi$ in Lemma \ref{lem:Nonliearity}, $P_2$ stands for the perturbation terms
\begin{align}\label{eqn:holder:p2}
P_2= \sum b_{ijk} \ast D^i \big( (J_{\m}-\lambda_0) \ast D^j J_{\m} \ast D^k J_{\m} \big)
\end{align}
where $b_{ijk}$ consists of $|g^{\widetilde{i}\widetilde{j}}-\delta^{\widetilde{i}\widetilde{j}}|$, $0\leq j,k \leq m$ and $i+j+k=2m$, and $\widetilde{L}_2$ stands for the lower order terms in the following form
\begin{align}
\widetilde{L}_2=\sum D^{i_1}g \ast \cdots \ast D^{i_s}g
\ast D^j J_{\m} \ast D^k J_{\m} \ast D^l J_{\m}
\end{align}
with $i_{\mu} \geq 1$ $\mu=1, \cdots, s$, $0\leq j+k+l\leq 2m-1$ and $\big(\sum_{\mu=1}^s i_{\mu} \big)+ j+k+l=2m$.

By above arguments, we get the final reduced equation about $J_{\m}$
\begin{align}\label{eqn:holder:temp1}
\Delta_0^m J_{\m}= T_s + \mathcal{P} + \mathcal{L}
\end{align}
where $\mathcal{P}=P_2 - P_1$ and $\mathcal{L}=\widetilde{L}_2 - \widetilde{L}_1$. In other words, the nonlinear part of (\ref{eqn:holder:temp1}) consists of three types of terms: terms that admit a good divergence form, the perturbation terms and the lower order terms.

Now let us recall the definition of $E(u,r)$ and $D_p(u,r)$ in (\ref{eqn:normlised energy}) and (\ref{eqn:Dp}) respectively. Then we claim that for any given $\tau \in (0,1)$, there exists $\delta_0>0$, $\epsilon_0>0$
and $\theta_0 \in (0,\frac{1}{2})$ such that if the metric $g$ satisfies (\ref{assumption:g}) and $E(J_{\m},1)<\epsilon_0$, then we have
\begin{align}\label{ineqy:holder:decay}
D_{p_0}(J_{\m}, \theta_0) \leq \theta_0^\tau \big(D_{p_0}(J_{\m}, 1)+ \|Dg\|_{C^{2m-1}(B_1)} \big).
\end{align}
where $p_0=\frac{4m}{3}$.
It is obvious that above claim is a direct consequence of Lemma \ref{lem:decay estimates} provided $\mathcal{P} \equiv 0$. Since the key point in the proof of Lemma \ref{lem:decay estimates} is that the inequality (\ref{lem:decay:main}) holds, it suffices to prove a similar inequality for additional nonlinear terms $\mathcal{P}$.

We claim that, if we denote $G(x)$ to be the fundamental solution for $\Delta^m_0$ on $\R^{2m}$ and do the similar arguments as that in the proof of Lemma \ref{lem:decay estimates}, with $u$ replaced by $J_{\m}$, then we have
\begin{align}\label{ineqy:holder:p}
\| D \omega_{\mathcal{P}} \|_{L^{p_0}(B_1)}
\leq C \big( \delta_0 \| D J_{\m} \|_{L^{p_0}(B_1)}
+ E(J_{\m},1) \|Dg \|_{C^{2m-1}(B_1)} \big),
\end{align}
where $C$ is a positive constant independent of $J_{\m}$ and
\begin{align*}
\omega_{\mathcal{P}}(x):= \int_{\R^{2m}} G(x-y) \mathcal{P}(\widetilde{J}_{\m})(y) dy.
\end{align*}
We now turn to proving (\ref{ineqy:holder:p}). Let us first deal with the perturbation terms with the form $a(x) D^{2m}J_{\m}$ (see (\ref{eqn:Holder:p1}))  in $\mathcal{P}$ where $|a(x)|\leq C \delta_0$.
Since
\begin{align*}
a(x) D^{2m} J_{\m}= D^{2m-1} \big( a(x) DJ_{\m} \big) + \mbox{Lower order terms}
\end{align*}
and
\begin{align*}
& \bigg \| D\int_{\R^{2m}} D^{2m-1}G(x-y) \widetilde{a}(y) D \widetilde{J}_{\m}(y) dy
\bigg \|_{L^{p_0}(\R^{2m})} \\
=& \bigg \| \int_{\R^{2m}} D^{2m}G(x-y) \widetilde{a}(y) D \widetilde{J}_{\m}(y) dy
\bigg \|_{L^{p_0}(\R^{2m})} \\
\leq & C \| \widetilde{a}(x) D \widetilde{J}_{\m}(x) \|_{L^{p_0}(\R^{2m})} \\
\leq & C \|a(x)\|_{L^\infty(B_1)} \|D J_{\m}(x) \|_{L^{p_0}(B_1)} \\
\leq & C \,\delta_0 \, \|D J_{\m}(x) \|_{L^{p_0}(B_1)},
\end{align*}
it follows from estimates for lower order terms in Lemma \ref{lem:decay estimates} that (\ref{ineqy:holder:p}) holds for the terms in the form $a(x) D^{2m} J_{\m}$. In the same manner,  (\ref{ineqy:holder:p}) also holds for the terms in (\ref{eqn:holder:p2}). Hence, our claim holds.

Finally, following the similar argument in the proof of Theorem \ref{thm:Holder} in the Euclidean case $(B_1,g_0)$, the decay estimate (\ref{ineqy:holder:decay}) implies $J_{\m} \in C^{0,\alpha}(B_1)$ for some $\alpha \in (0,1)$. The proof is complete.
\qed

\section{Higher regularity for \texorpdfstring{$m$}{m}-harmonic almost complex structures}\label{sec:smooth}
In this section, we will apply Theorem \ref{thm:higher regularity} to prove the following theorem.
\begin{theorem}\label{thm:J:smooth}
Suppose $n\geq 2m$ $(m \geq 1)$ and $J \in C^{0,\alpha} \cap W^{m,2}$ is a weakly $m$-harmonic almost complex structure on $(M^n,g)$. Then $J$ is smooth.
\end{theorem}

\begin{proof}
Since the higher regularity is a local property in nature, we assume that $(M,g)$ is a unit ball $B_1$ in $\R^n$ equipped with a smooth metric $g \in C^\infty(\overline{B}_1)$.

Firstly, it follows from Proposition \ref{prop:EL-div} that any $m$-harmonic almost complex structure $J$ satisfies the following equation in distributional sense
\begin{align}\label{eqn:smooth:el}
\Delta^m J=\sum_{s=0}^{m-1}(-1)^{m+1+s} \nabla^s \cdot g_s,
\end{align}
where
\begin{align*}
g_s =\sum_{\substack {k_1+k_2+k_3=2m-s \\ 0 \leq k_1, k_2, k_3 \leq m}}
C(k_1, k_2, k_3)\,
\nabla^{k_1} J \nabla^{k_2} J \nabla^{k_3} J,
\quad C(k_1,k_2,k_3)\in \mathbb{Z}.
\end{align*}

Next, by the same argument as in the proof of Theorem \ref{thm:Holder} in the general case $(B_1,g)$, we can reduce (\ref{eqn:smooth:el}) to the following equation
\begin{align}\label{eqn:smooth:reduce}
\Delta^m J_{\m}= \sum_{s=0}^{m-1}(-1)^{m+1+s} \nabla^s \cdot \widetilde{g}_s + \mathcal{L},
\end{align}
where $\widetilde{g}_s$ is the main body of nonlinearity in the following form
\begin{align*}
\widetilde{g}_s = \sum C_{k_1k_2k_3} \nabla^{k_1} J_{\m}\, \nabla^{k_2} J_{\m} \,\nabla^{k_3} J_{\m},
\end{align*}
where $C_{k_1k_2k_3}$ are constants and $k_1+k_2+k_3=2m-s$, and $\mathcal{L}$ contains the lower order terms in the following form
\begin{align*}
\mathcal{L}
=\sum D^{i_1}g \ast \cdots \ast D^{i_s} g\ast
\big( D^j J_{\m} + D^{k_1} J_{\m} \ast D^{k_2}  J_{\m} \ast D^{k_3} J_{\m} \big)
\end{align*}
with $i_\mu \geq 1, \mu=1, \cdots, s$, $j=k_1+k_2+k_3$ and $(\sum_{\mu=1}^s i_{\mu}) +j=2m$.

Obviously, $J_{\m} \in C^\infty(B_1)$ follows from Theorem \ref{thm:higher regularity} provided metric $g$ is Euclidean. Note that all tools we employ in proving Theorem \ref{thm:higher regularity} are Sobolev inequalities, Poincar\'e inequalities, interpolation inequalities, Sobolev extension Theorems, Green functions and a priori estimates for elliptic equations on Euclidean spaces. Since there are corresponding versions of above-mentioned tools on Riemannian manifold $(B_1,g)$ with metric $g\in C^\infty(\overline{B}_1)$, we can obtain $J_{\m} \in C^\infty(B_1)$ for the general case $g\in C^\infty(\overline{B}_1)$ by following the same arguments as in the proof of Theorem \ref{thm:higher regularity}. The details are left to the reader.
\end{proof}


\section{Appendix}\label{sec:appendix}
In this section, we will rewrite $m$-harmonic almost complex structure equation in a good divergence form.

\medskip
We first derive an equivalent form of $m$-harmonic almost complex structure equation.
\begin{theorem}\label{thm:el:nonliearity}
The Euler-Lagrange equation with respect to the functional $\mathcal{E}_{m}(J)$
\begin{align*}
[\Delta^{m}J,J]=0
\end{align*}
is equivalent to
\begin{align*}
\Delta^{m}J=\frac{1}{2}JQ_m=\frac{1}{2}Q_m J
\end{align*}
where
\begin{align*}
Q_m = \Delta^m(J^2)-\Delta^m J \,J- J\, \Delta^m J.
\end{align*}
Hence $m$-harmonic almost complex structures satisfies the following equation
\begin{align}\label{eqn:EL:m-poly}
4 \Delta^m J= T_m(J,\nabla J, \cdots, \nabla^{2m-1}J)
\end{align}
where $T_m=JQ_m+Q_mJ$.
\end{theorem}
\begin{proof}
Since the proof is just making simple derivative computations, we give it only for the case $m=2$.

For every almost complex structure $J$, we have $J^2=-id$ and $\Delta (id)=0$. Then we compute
\begin{align*}
0=\Delta^2 (J^2)
&=\Delta(\Delta J J +2 \nabla J \nabla J + J \Delta J) \\
&=\Delta^2 J J + 2 \nabla \Delta J \nabla J  + 2\nabla J \nabla \Delta J+ 2\Delta J \Delta J
+ 2 \Delta( \nabla J )^2
+ J \Delta^2 J.
\end{align*}
Let
\begin{align*}
Q_2
:=2 \nabla \Delta J \nabla J  + 2\nabla J \nabla \Delta J+ 2\Delta J \Delta J
+ 2 \Delta( \nabla J )^2,
\end{align*}
and  we obtain
\begin{align*}
Q_2=\Delta^2 (J^2)-\Delta^2 J J - J \Delta^2 J=-\Delta^2 J J - J \Delta^2 J,
\end{align*}
which implies that
\begin{align*}
JQ_2=Q_2J=\Delta^2 J - J \Delta^2 J J.
\end{align*}
Clearly, $[\Delta^2 J, J]=0$ is equivalent to
\begin{align*}
JQ_2=Q_2J=\Delta^2 J - J \Delta^2 J J=2\Delta^2 J.
\end{align*}
which completes the proof.

\end{proof}

Now let us state our main conclusion in this section.
\begin{proposition}\label{prop:T_lambda}
Suppose $m\in \{1,2,3 \}$ and suppose $B_1 \subset \R^n$ is a unit Euclidean ball and $g$ is a smooth metric tensor on $B_1$. If $J$ is a smooth square matrix valued function and satisfies $J^2=-id$, then for every fixed constant matrix $\lambda_0$, $T_m$ defined in Theorem \ref{thm:el:nonliearity} can be rewritten in the following form
\begin{align}
T_m=T_{\lambda_0} - \big[J-\lambda_0, [\Delta^m J, J]\big]
\end{align}
where $T_{\lambda_0}$ is a linear combination of the following terms
\begin{align*}
\nabla^{\alpha}\bigg((J-\lambda_{0})\ast\nabla^{\beta}J\ast\nabla^{\gamma}J\bigg)
\quad or\quad
\lambda_0\ast\nabla^{\alpha}\bigg((J-\lambda_{0})\ast\nabla^{\delta}J\bigg),
\end{align*}
where $\alpha,\beta,\gamma,\delta$ are multi-indices such that $1\leq|\alpha| \leq 2m-1$, $0\leq |\beta|,|\gamma|,|\delta|\leq m$, $|\alpha|+|\beta|+|\gamma|=2m$
and $|\alpha|+|\delta|=2m$.

\end{proposition}
We will prove above proposition in the following three subsections. In what follows, we always assume $J$ is a square matrix valued function and satisfies $J^2=-id$. In this situation, we know $\nabla \lambda_0=0$ for every constant matrix $\lambda_0$. The reason for emphasizing this point is that if we consider a constant matrix $\lambda_0 \in M_n(\R)$ as a (1,1) tensor field on $(B_1,g)$, then (1,2) tensor field $\nabla \lambda_0 \neq 0$ due to the metric $g$.

\subsection{The case m=1: harmonic almost complex structure}
By the definition of $T_m$ in Theorem \ref{thm:el:nonliearity}, we have
\begin{align*}
T_1=2 J (\nabla J)^2 + 2 (\nabla J)^2 J.
\end{align*}
The fact $J^2=-id$ implies that
\begin{align}\label{eqn:m1:temp0}
\nabla J J + J \nabla J=0.
\end{align}
Thus, we obtain
\begin{align*}
J (\nabla J )^2 + (\nabla J )^2 J
= J (\nabla J )^2 -\nabla J J \nabla J
= [J, \nabla J] \nabla J.
\end{align*}
On the other hand, there also holds
\begin{align*}
J (\nabla J )^2 + (\nabla J )^2 J
= - \nabla J J \nabla J  + (\nabla J )^2 J
= \nabla J [\nabla J,  J].
\end{align*}
Hence, it follows that
\begin{align}\label{eqn:m1:temp1}
T_1= [J, \nabla J] \nabla J + \nabla J [\nabla J,  J]
=[\nabla J, [\nabla J, J]].
\end{align}
Since
\begin{align*}
\nabla [\nabla J, J]= [\Delta J, J],
\end{align*}
we see that
\begin{align}\label{eqn:m1:temp2}
\nabla [J-\lambda_0, [\nabla J, J]]
=[\nabla J, [\nabla J, J]]+ [J-\lambda_0, [\Delta J, J]],
\end{align}
Combining (\ref{eqn:m1:temp1} and (\ref{eqn:m1:temp2}), we conclude that
\begin{align*}
T_1= \nabla [J-\lambda_0, [\nabla J, J]] - [J-\lambda_0, [\Delta J, J]]
=T_{\lambda_0}- [J-\lambda_0, [\Delta J, J]],
\end{align*}
which is the desired conclusion.

\subsection{The case m=2: biharmonic almost complex structure}
By the definition of $T_m$ in Theorem \ref{thm:el:nonliearity}, we have
\begin{align*}
T_2= J Q_2 + Q_2 J,
\end{align*}
where
\begin{align*}
Q_2=2 \nabla \Delta J \nabla J  + 2\nabla J \nabla \Delta J+ 2\Delta J \Delta J
+ 2 \Delta( \nabla J )^2.
\end{align*}
Set
\begin{align*}
\mathbf{I}= & J \big( \nabla \Delta J \nabla J + \nabla J \nabla \Delta J \big)
+ \big( \nabla \Delta J \nabla J + \nabla J \nabla \Delta J \big) J \\
\mathbf{II}= & J (\Delta J)^2 + (\Delta J)^2 J\\
\mathbf{III}= & J \Delta( \nabla J )^2 + \Delta( \nabla J )^2 J.
\end{align*}
Thus, we obtain
\begin{align*}
T_2=2\mathbf{I}+2\mathbf{II}+2\mathbf{III}.
\end{align*}
Firstly, we compute the term $\mathbf{I}$:
\begin{align}\label{eqn:m2:I}
\mathbf{I}=& J \nabla \Delta J \nabla J + J \nabla J \nabla \Delta J
+  \nabla \Delta J \nabla J J+ \nabla J \nabla \Delta J  J \nonumber \\
=& J \nabla \Delta J \nabla J - \nabla J J \nabla \Delta J
- \nabla \Delta J J \nabla J + \nabla J \nabla \Delta J  J  \nonumber\\
=& [J, \nabla \Delta J] \nabla J + \nabla J [\nabla \Delta J, J]  \nonumber\\
=& [\nabla J, [\nabla \Delta J, J]].
\end{align}

Since
\begin{align*}
\nabla \bigg( [\nabla \Delta J, J] -[\Delta J, \nabla J] \bigg)= [\Delta^2 J, J],
\end{align*}
we have
\begin{align} \label{eqn:m2:temp1}
& \nabla [J -\lambda_0, [\nabla \Delta J, J] -[\Delta J, \nabla J]] \nonumber \\
=& [\nabla J, [\nabla \Delta J, J] -[\Delta J, \nabla J]]+ [J-\lambda_0, [\Delta^2 J, J]].
\end{align}
Now we compute the left-hand side of above equality:
\begin{align*}
& \nabla [J -\lambda_0, [\nabla \Delta J, J] -[\Delta J, \nabla J]] \\
=& \nabla  [J -\lambda_0, \nabla [\Delta J, J]]
-2 \nabla [J -\lambda_0, [\Delta J, \nabla J]] \\
=& \nabla  [J -\lambda_0, \nabla [\Delta J, J]]+ T_{\lambda_0}\\
=& \Delta [J -\lambda_0, [\Delta J, J]] - \nabla [\nabla J, [\Delta J, J]]
+ T_{\lambda_0}\\
=& - \nabla [\nabla J, [\Delta J, J]] + T_{\lambda_0}\\
=& - [\Delta J, [\Delta J, J]] -[\nabla J, [\nabla \Delta J, J]] -[\nabla J, [\Delta J, \nabla J]] + T_{\lambda_0}
\end{align*}
Substituting above equality into (\ref{eqn:m2:temp1}) yields
\begin{align}\label{eqn:m2:temp2}
2 \mathbf{I}= 2 [\nabla J, [\nabla \Delta J, J]]
= - [\Delta J, [\Delta J, J]] - [J-\lambda_0, [\Delta^2 J, J]] + T_{\lambda_0}
\end{align}
We now turn to compute the term $\mathbf{III}$. Since
\begin{align*}
J \Delta (\nabla J)^2
=& (J -\lambda_0 ) \Delta (\nabla J)^2 + \lambda_0 \Delta (\nabla J)^2 \\
=& (J -\lambda_0 ) \Delta (\nabla J)^2
+ \lambda_0 \Delta \nabla \big( ( J -\lambda_0) \nabla J \big)
- \lambda_0 \Delta \big( (J-\lambda_0) \Delta J \big) \\
=&(J -\lambda_0 ) \Delta (\nabla J)^2  + T_{\lambda_0} \\
=& \nabla_p \big( (J -\lambda_0 ) \nabla_p (\nabla J)^2 \big)
-\nabla_p J \nabla_p (\nabla J)^2  + T_{\lambda_0} \\
=& -\nabla_p J \nabla_p (\nabla J)^2 + T_{\lambda_0} \\
=& -\nabla_p \big( \nabla_p J (\nabla J)^2 \big)
+ \Delta J (\nabla J)^2  + T_{\lambda_0} \\
=& \Delta J (\nabla J)^2 + T_{\lambda_0}
\end{align*}
and similarly
\begin{align*}
\Delta (\nabla J)^2 J = (\nabla J)^2 \Delta J + T_{\lambda_0},
\end{align*}
we have
\begin{align}\label{eqn:m2:III}
\mathbf{III}= \Delta J (\nabla J)^2 + (\nabla J)^2 \Delta J + T_{\lambda_0}
\end{align}
Now let us proceed to compute $\mathbf{II}$:
\begin{align*}
\mathbf{II}=& J (\Delta J)^2 + (\Delta J)^2 J \\
=& - \big( \Delta J J + 2( \nabla J )^2 \big) \Delta J + \Delta J \Delta J J \\
=& \Delta J [\Delta J, J] - 2 (\nabla J )^2 \Delta J
\end{align*}
where we used the fact $0=\Delta( J^2)=\Delta J J +2 (\nabla J)^2 + J \Delta J$.
On the other hand, we also have
\begin{align*}
\mathbf{II}=& J (\Delta J)^2 + (\Delta J)^2 J \\
=& J \Delta J \Delta J - \Delta J \big( J \Delta J + 2 (\nabla J)^2 \big) \\
=& [J, \Delta J] \Delta J - 2 \Delta J (\nabla J)^2
\end{align*}
Hence, we obtain
\begin{align}\label{eqn:m2:II}
2 \mathbf{II} =& \Delta J [\Delta J, J]+ [J, \Delta J] \Delta J
- 2\big( (\nabla J )^2 \Delta J+ \Delta J (\nabla J)^2\big) \nonumber \\
=& [\Delta J, [\Delta J, J]]
- 2(\nabla J )^2 \Delta J-2 \Delta J (\nabla J)^2 \nonumber \\
=& [\Delta J, [\Delta J, J]] - 2 \mathbf{III}
\end{align}
where in the last equality we used (\ref{eqn:m2:III}). Substituting (\ref{eqn:m2:II}) into (\ref{eqn:m2:temp2}), we get
\begin{align*}
T_2=2\mathbf{I}+ 2\mathbf{II} +2\mathbf{III}
= T_{\lambda_0}-[J-\lambda_0, [\Delta^2 J, J]]
\end{align*}
which is the desired conclusion.

\subsection{The case m=3: 3-harmonic almost complex structure}
By the definition of $T_m$ in Theorem \ref{thm:el:nonliearity}, we have
\begin{align*}
T_3= J Q_3 + Q_3 J,
\end{align*}
where
\begin{align*}
Q_3 =& 2\nabla\Delta^{2}J\nabla J+2\nabla J\nabla\Delta^{2}J
+\Delta^{2}J\Delta J+\Delta J\Delta^{2}J \nonumber \\
& +2\Delta \big(\nabla\Delta J\nabla J+\nabla J\nabla\Delta J \big)
+2\Delta\left(\Delta J\right)^{2}+2\Delta^{2}\left(\nabla J\right)^{2}.
\end{align*}

For simplicity, we collect some terms which are $T_{\lambda_0}$ type and appear frequently in the following proof.
\begin{lemma}\label{lem:m3:div}
The following terms are $T_{\lambda_0}$ type terms for any given constant matrix $\lambda_0$:
\begin{align*}
\nabla\left(\nabla J\ast\nabla^{2}J\ast\nabla^{2}J\right),\,
\nabla^{2}\left(\nabla J\ast\nabla J\ast\nabla^{2}J\right),\,
\nabla\left(\nabla J\ast\nabla J\ast\nabla^{3}J\right),\,
\nabla^4 \left(\nabla J\right)^2.
\end{align*}

\end{lemma}

\begin{proof}
The proof is quite easy. For simplicity, we only show how to rewrite
the first term and the third term. Other terms can be handled in much the same way.\\
The first term:
\begin{align*}
& \nabla \big(\nabla J\ast\nabla^{2}J\ast\nabla^{2}J \big) \\
=&\nabla \left(\nabla \left(J-\lambda_{0}\right) \ast \nabla^{2}J \ast \nabla^{2}J \right)\\
=&\nabla^{2} \left( \left(J-\lambda_{0}\right) \ast \nabla^{2}J \ast \nabla^{2}J\right)
-\nabla \left( \left(J-\lambda_{0}\right) \ast \nabla^{3}J \ast \nabla^{2}J \right)
-\nabla \left( \left(J-\lambda_{0}\right) \ast \nabla^{2}J \ast \nabla^{3}J \right)\\
=& T_{\lambda_0}.
\end{align*}
The third term:
\begin{align*}
& \nabla \big( \nabla J \ast \nabla J \ast \nabla^{3}J \big) \\
=& \nabla^2 \big( \nabla J \ast \nabla J \ast \nabla^2 J \big)
- \nabla \big( \nabla^2 J \ast \nabla J \ast \nabla^{2}J \big)
-\nabla \big( \nabla J \ast \nabla^2 J \ast \nabla^{2}J \big) \\
=& T_{\lambda_0}.
\end{align*}

\end{proof}

Note that we will emphasize the terms of $T_{\lambda_0}$ type by underlining it in the following proof. Set
\begin{align*}
\mathbf{I}= & \,J \nabla \Delta^{2} J \nabla J + J\nabla J \nabla\Delta^{2}J
+\nabla \Delta^{2}J\nabla J\, J+\nabla J \nabla\Delta^{2}J \,J, \\
\mathbf{II}= &\, J \big(\Delta^{2}J\Delta J+\Delta J\Delta^{2}J \big)
+\big( \Delta^{2}J\Delta J+\Delta J\Delta^{2}J \big)J, \\
\mathbf{III}=& \, J\Delta \big( \nabla\Delta J\nabla J+\nabla J\nabla\Delta J \big)
+\Delta \big(\nabla\Delta J\nabla J+\nabla J\nabla\Delta J \big)J, \\
\mathbf{IV}=& \, J\Delta\left(\Delta J\right)^{2}
+\Delta\left(\Delta J\right)^{2}J, \\
\mathbf{V}=& \, J\Delta^{2}\left(\nabla J\right)^{2}
+\Delta^{2}\left(\nabla J\right)^{2}J.
\end{align*}
Then, we obtain
\begin{align*}
T_3= 2\mathbf{I}+ \mathbf{II}+  2\mathbf{III}+ 2\mathbf{IV}+ 2\mathbf{V}.
\end{align*}

\noindent \textbf{Step One: dealing with I}
\medskip

Now Let us compute the first term \textbf{I}:
\begin{align}\label{eqn:m3:I}
\mathbf{I} & =J\nabla\Delta^{2}J\nabla J + J\nabla J\nabla\Delta^{2}J
+ \nabla\Delta^{2}J\nabla J J + \nabla J\nabla\Delta^{2}J J \nonumber \\
& =J\nabla\Delta^{2}J\nabla J-\nabla JJ\nabla\Delta^{2}J
-\nabla\Delta^{2}JJ\nabla J + \nabla J\nabla\Delta^{2}J J \nonumber \\
& =[J,\nabla\Delta^{2}J]\nabla J + \nabla J[\nabla\Delta^{2}J,J] \nonumber \\
& =[\nabla J,[\nabla\Delta^{2}J,J]],
\end{align}
where in the second equality we use the fact $\nabla (J^2) =0$ which implies $\nabla J J=-J\nabla J$.

Since
\begin{align*}
\nabla \bigg([\nabla\Delta^{2}J,J]-[\Delta^{2}J,\nabla J]
+[\nabla\Delta J,\Delta J] \bigg)=[\Delta^3 J, J],
\end{align*}
we have
\begin{align}\label{eqn:m3:temp1}
& \nabla \big[J-\lambda_{0},[\nabla\Delta^{2}J,J]-[\Delta^{2}J,\nabla J]
+[\nabla\Delta J,\Delta J] \big] \nonumber \\
= & \big[\nabla J,[\nabla\Delta^{2}J,J]-[\Delta^{2}J,\nabla J]+[\nabla\Delta J,\Delta J]\big]
+ \big[ J-\lambda_{0}, [\Delta^3 J, J]\big].
\end{align}
Now we compute the left-hand side of above equality.
\begin{align*}
& \nabla \big[J-\lambda_{0},[\nabla\Delta^{2}J,J]-[\Delta^{2}J,\nabla J]
+\underline{[\nabla\Delta J,\Delta J]} \big]\\
=& \nabla \big[J-\lambda_{0},[\nabla\Delta^{2}J,J]-[\Delta^{2}J,\nabla J] \big]
+{T_{\lambda_0}}\\
=& \nabla[J-\lambda_{0},\nabla[\Delta^{2}J,J]-2[\Delta^{2}J,\nabla J]]
+{T_{\lambda_0}}\\
=& \Delta \big[J-\lambda_{0},[\Delta^{2}J,J]\big]
 - \nabla \big[\nabla J,[\Delta^{2}J,J]\big]
- 2\nabla[J-\lambda_{0},[\Delta^{2}J,\nabla J]] + T_{\lambda_0}\\
=& \Delta \big[J-\lambda_{0},\nabla[\nabla\Delta J,J]
-\underline{[\nabla\Delta J,\nabla J]}\big]\\
 & -2\nabla_p \big[J-\lambda_{0},\nabla_q[\nabla_q\Delta J,\nabla_p J]
 -\underline{[\nabla_q\Delta J,\nabla_{qp}^2 J]} \big]\\
 & {-\nabla[\nabla J,[\Delta^{2}J,J]]}+ T_{\lambda_0}\\
=&\underline{\Delta\nabla[J-\lambda_{0},[\nabla\Delta J,J]]}
 -\Delta[\nabla J,[\nabla\Delta J,J]]\\
 & -2 \underline{\nabla_{pq}^{2} \big[ J-\lambda_{0},[\nabla_q \Delta J,\nabla J] \big]}
 +2 \underline{\nabla_{p}[\nabla_{q}J,[\nabla_{q}\Delta J,\nabla_{p}J]]}\\
 & {-\nabla[\nabla J,[\Delta^{2}J,J]]}+T_{\lambda_0}\\
=& -\Delta[\nabla J,[\nabla\Delta J,J]]
-\nabla[\nabla J,[\Delta^{2}J,J]]+ T_{\lambda_0},
\end{align*}
where in the second equality from bottom we employ lemma \ref{lem:m3:div}.
By substituting above equality into (\ref{eqn:m3:temp1}), we obtain
\begin{align*}
\mathbf{I} =& [\nabla J,[\nabla\Delta^{2}J,J]]\\
=& [\nabla J,[\Delta^{2}J,\nabla J]]-[\nabla J,[\nabla\Delta J,\Delta J]]-\Delta[\nabla J,[\nabla\Delta J,J]]-\nabla[\nabla J,[\Delta^{2}J,J]] \\
& +{T_{\lambda_0}} - [J-\lambda_0, [\Delta^3 J, J]].
\end{align*}
Since
\[
\nabla[\nabla J,[\Delta^{2}J,J]]=[\Delta J,[\Delta^{2}J,J]]+[\nabla J,[\nabla\Delta^{2}J,J]]+[\nabla J,[\Delta^{2}J,\nabla J]],
\]
we deduce
\begin{align}\label{eqn:m3:temp2}
2\mathbf{I}= &-[\nabla J,[\nabla\Delta J,\Delta J]]-\Delta[\nabla J,[\nabla\Delta J,J]]-[\Delta J,[\Delta^{2}J,J]] \nonumber \\
& +{T_{\lambda_0}}- [J-\lambda_0, [\Delta^3 J, J]].
\end{align}
By lemma \ref{lem:m3:div}, we can derive
\begin{align}
& [\nabla J,[\nabla\Delta J,\Delta J]] \nonumber \\
=& \nabla J\left(\nabla\Delta J\Delta J-\Delta J\nabla\Delta J\right)
-\left(\nabla\Delta J\Delta J-\Delta J\nabla\Delta J\right)\nabla J \nonumber \\
=& \nabla J\nabla\Delta J\Delta J+\Delta J\nabla\Delta J\nabla J
-\nabla J\Delta J\nabla\Delta J-\nabla\Delta J\Delta J\nabla J \nonumber \\
=& \underline{\nabla\left(\nabla J\Delta J\Delta J\right)}
+\underline{\nabla\left(\Delta J\Delta J\nabla J\right)}
-2\left(\Delta J\right)^{3}
-2\nabla J\Delta J\nabla\Delta J
-2\nabla\Delta J\Delta J\nabla J \nonumber \\
=& -2\left(\Delta J\right)^{3}
-2\nabla J\Delta J\nabla\Delta J
-2\nabla\Delta J\Delta J\nabla J
+{T_{\lambda_0}} \label{eqn:m3:temp3}
\end{align}
and
\begin{align}
& \Delta[\nabla J,[\nabla\Delta J,J]] \nonumber \\
=& \Delta \big[\nabla J,\nabla[\Delta J,J]-[\Delta J,\nabla J] \big] \nonumber\\
=& \Delta[\nabla J,\nabla[\Delta J,J]]
-\underline{\Delta[\nabla J,[\Delta J,\nabla J]]} \nonumber \\
=& \underline{\Delta\nabla[\nabla J,[\Delta J,J]]}
-\Delta[\Delta J,[\Delta J,J]]+{T_{\lambda_0}}\nonumber \\
=& -\Delta\left(\left(\Delta J\right)^{2}J+J\left(\Delta J\right)^{2}-2\Delta JJ\Delta J\right)+{T_{\lambda_0}} \nonumber\\
=& -\Delta\left( 2\left(\Delta J\right)^{2}J
+2J\left(\Delta J\right)^{2}
+2\underline{\left(\nabla J\right)^{2}\Delta J}
+2\underline{\Delta J\left(\nabla J\right)^{2}} \right)
+T_{\lambda_0}\nonumber \\
=& -2\Delta\left(\left(\Delta J\right)^{2}J+J\left(\Delta J\right)^{2}\right)+{T_{\lambda_0}}. \label{eqn:m3:temp4}
\end{align}
where in the second equality from bottom we used the fact that
$$0=\Delta(J^2)=\Delta J J + 2 (\nabla J)^2 + J \Delta J.$$
Substituting equalities (\ref{eqn:m3:temp3}) and (\ref{eqn:m3:temp4}) into equality (\ref{eqn:m3:temp2}) yields
\begin{align}
2\mathbf{I} =& 2\left(\Delta J\right)^{3}+2\nabla J\Delta J\nabla\Delta J+2\nabla\Delta J\Delta J\nabla J+2\Delta\left(\left(\Delta J\right)^{2}J+J\left(\Delta J\right)^{2}\right) \nonumber \\
&-[\Delta J,[\Delta^{2}J,J]]
+{T_{\lambda_0}} - [J-\lambda_0, [\Delta^3 J, J]].\label{eqn:m3:key}
\end{align}

\bigskip
\noindent \textbf{Step Two: dealing with V and II}
\medskip

Firstly, we deal with fifth term \textbf{V}. It follows from Lemma \ref{lem:m3:div} that
\begin{align*}
\mathbf{V}
& =J\Delta^{2}\left(\nabla J\right)^{2}+\Delta^{2}\left(\nabla J\right)^{2}J\\
& =\left(J-\lambda_{0}\right)\Delta^{2}\left(\nabla J\right)^{2}
+\Delta^{2}\left(\nabla J\right)^{2}\left(J-\lambda_{0}\right)
+\underline{\lambda_{0}\Delta^{2}\left(\nabla J\right)^{2}}
+\underline{\Delta^{2}\left(\nabla J\right)^{2}\lambda_{0}}\\
& =\nabla\left(\left(J-\lambda_{0}\right)\nabla\Delta\left(\nabla J\right)^{2}\right)
-\nabla J\nabla\Delta\left(\nabla J\right)^{2} \\
&\quad +\nabla\left(\nabla\Delta\left(\nabla J\right)^{2}\left(J-\lambda_{0}\right)\right)
 -\nabla\Delta\left(\nabla J\right)^{2}\nabla J
 +{T_{\lambda_0}}\\
& =\underline{\Delta\bigg((J-\lambda_{0})\Delta(\nabla J)^{2}\bigg)}
-\nabla\left(\nabla J\Delta\left(\nabla J\right)^{2}\right)
-\nabla J\nabla\Delta\left(\nabla J\right)^{2} \\
& \quad + \underline{\Delta\bigg(\Delta(\nabla J)^{2}(J-\lambda_{0})\bigg)}
-\nabla\left(\Delta\left(\nabla J\right)^{2}\nabla J\right)
-\nabla\Delta\left(\nabla J\right)^{2}\nabla J
+{T_{\lambda_0}}\\
& =-\nabla\left(\nabla J\Delta\left(\nabla J\right)^{2}\right)
-\nabla J\nabla\Delta\left(\nabla J\right)^{2}
-\nabla\left(\Delta\left(\nabla J\right)^{2}\nabla J\right)
-\nabla\Delta\left(\nabla J\right)^{2}\nabla J
+{T_{\lambda_0}}.
\end{align*}
 Since
\[
\nabla_{p}\bigg(\nabla_{p}J\Delta\left(\nabla J\right)^{2}\bigg)
=\nabla_{pq}^2\bigg(\nabla_{p}J\nabla_{q}\left(\nabla J\right)^{2}\bigg)
-\nabla_{p}\bigg( \nabla_{qp}^{2}J\nabla_{q}\left(\nabla J\right)^{2}\bigg)
={T_{\lambda_0}}
\]
and
\begin{align*}
\nabla_{p}J\nabla_{p}\Delta\left(\nabla J\right)^{2}
& =\underline{\nabla_{p}\bigg(\nabla_{p}J\Delta\left(\nabla J\right)^{2}\bigg)}
-\Delta J\Delta\left(\nabla J\right)^{2}\\
& =-\underline{\nabla_{p}\left(\Delta J\nabla_{p}\left(\nabla J\right)^{2}\right)}
+\nabla_{p}\Delta J\nabla_{p}\left(\nabla J\right)^{2}+{T_{\lambda_0}}\\
& =\underline{\nabla_{p}\left(\nabla_{p}\Delta J\left(\nabla J\right)^{2}\right)}
-\Delta^{2}J\left(\nabla J\right)^{2}
+{T_{\lambda_0}}\\
 & =-\Delta^{2}J\left(\nabla J\right)^{2}+{T_{\lambda_0}},
\end{align*}
we have
\begin{equation}\label{eqn:m3:V}
\mathbf{V}=\Delta^{2}J\left(\nabla J\right)^{2}+\left(\nabla J\right)^{2}\Delta^{2}J+{T_{\lambda_0}}.
\end{equation}
Next, we deal with the second term
\begin{align}
\mathbf{II} & =J\left(\Delta^{2}J\Delta J+\Delta J\Delta^{2}J\right)+\left(\Delta^{2}J\Delta J+\Delta J\Delta^{2}J\right)J\nonumber \\
 & =[\Delta J,[\Delta^{2}J,J]]-2\left(\nabla J\right)^{2}\Delta^{2}J-2\Delta^{2}J\left(\nabla J\right)^{2} \nonumber \\
 & =[\Delta J,[\Delta^{2}J,J]]-2\mathbf{V}+T_{\lambda_0}\label{eqn:m3:II},
\end{align}
where in the second equality we used the fact $\Delta (J^2)=0$ which implies
\[
\Delta J J=-J \Delta J- 2 \nabla J \nabla J,
\]
and in the last equality we used (\ref{eqn:m3:V}).

\bigskip
\noindent \textbf{Step Three: dealing with III }
\medskip

Here we begin to deal with the third term:
\begin{align*}
\mathbf{III}
& =J\Delta\bigg(\nabla\Delta J\nabla J+\nabla J\nabla\Delta J\bigg)
+\Delta\bigg(\nabla\Delta J\nabla J+\nabla J\nabla\Delta J\bigg)J\\
& =J\Delta\left(\nabla \big(\Delta J\nabla J+\nabla J\Delta J\big)
-2\left(\Delta J\right)^{2}\right)\\
&\quad \, +\Delta\left(\nabla\big(\Delta J\nabla J+\nabla J\Delta J\big)
 -2\left(\Delta J\right)^{2}\right)J\\
& =J\Delta\nabla\bigg(\Delta J\nabla J+\nabla J\Delta J\bigg)
+\Delta\nabla \bigg(\Delta J\nabla J+\nabla J\Delta J\bigg)J\\
&\quad \, -2\left(J\Delta\left(\Delta J\right)^{2}
+\Delta\left(\Delta J\right)^{2}J \right)\\
& =J\Delta \nabla\bigg(\Delta J\nabla J+\nabla J\Delta J\bigg)
+\Delta\nabla\bigg(\Delta J\nabla J+\nabla J\Delta J\bigg)J
-2\mathrm{\mathbf{IV}}.
\end{align*}
Since
\begin{align*}
& J\Delta\nabla\big(\Delta J\nabla J\big) \\
=& \big(J-\lambda_{0}\big)\Delta\nabla\left(\Delta J\nabla J\right)
+\lambda_{0}\Delta\nabla\left(\Delta J\nabla J\right)\\
=& \left(J-\lambda_{0}\right)\Delta\nabla\left(\Delta J\nabla J\right)
+\underline{\lambda_{0}\Delta\nabla
\bigg(\nabla\left(\Delta J\big(J-\lambda_{0}\big)\right)
-\nabla\Delta J\big(J-\lambda_{0}\big)\bigg)}\\
=& \left(J-\lambda_{0}\right)\Delta\nabla\big(\Delta J\nabla J\big)
+{T_{\lambda_0}}\\
=& \nabla_p \bigg( (J-\lambda_0) \nabla^2_{pq} \big( \Delta J \nabla_q J\big)\bigg)
-\nabla_p J \nabla^2_{pq} \big( \Delta J \nabla_q J\big)
+{T_{\lambda_0}}\\
=& \underline{\Delta \bigg( (J-\lambda_0) \nabla_q \big(\Delta J \nabla_q J \big)\bigg)}
-\underline{\nabla_p \bigg( \nabla_p J \nabla_q \big( \Delta J \nabla_q J\big)\bigg)}
-\nabla_p J \nabla^2_{pq} \big( \Delta J \nabla_q J\big)
+{T_{\lambda_0}}\\
=& - \underline{\nabla_p \bigg( \nabla_p J \nabla_q \big(\Delta J \nabla_q J \big)\bigg)}
+ \Delta J \nabla_q \big( \Delta J \nabla_q J \big)
+{T_{\lambda_0}}\\
=& \underline{\nabla_q \bigg( \Delta J \Delta J \nabla_q J \bigg)}
-\nabla \Delta J \Delta J \nabla J
+{T_{\lambda_0}}\\
=& -\nabla\Delta J\Delta J\nabla J+{T_{\lambda_0}},
\end{align*}
we have
\begin{align}
\mathbf{III}
& =-\nabla\Delta J\bigg(\Delta J\nabla J+\nabla J\Delta J\bigg)
-\bigg(\Delta J\nabla J+\nabla J\Delta J\bigg)\nabla\Delta J
-2\mathbf{IV}+{T_{\lambda_0}}\nonumber \\
& =-\nabla\Delta J\Delta J\nabla J-\nabla J\Delta J\nabla\Delta J
-\bigg(\nabla\Delta J\nabla J\Delta J+\Delta J\nabla J\nabla\Delta J\bigg)
-2\mathbf{IV}+{T_{\lambda_0}}\nonumber \\
& =-\nabla\Delta J\Delta J\nabla J
-\nabla J\Delta J\nabla\Delta J
-\underline{\nabla \big(\Delta J\nabla J\Delta J \big)}
+(\Delta J)^{3}
-2\mathbf{IV}+{T_{\lambda_0}}\nonumber \\
 & =-\nabla\Delta J\Delta J\nabla J-\nabla J\Delta J\nabla\Delta J+\left(\Delta J\right)^{3}-2\mathbf{IV}+{T_{\lambda_0}}.\label{eqn:m3:III}
\end{align}

\bigskip
\noindent \textbf{Step Four: dealing with IV}
\medskip

Since
\begin{align*}
J\Delta\left(\Delta J\right)^{2}
& =\nabla_{p}\left(J\nabla_{p}\left(\Delta J\right)^{2}\right)-\nabla_{p}J\nabla_{p}\left(\Delta J\right)^{2}\\
& =\Delta\left(J\left(\Delta J\right)^{2}\right)
-\underline{\nabla_{p}\left(\nabla_{p}J\left(\Delta J\right)^{2}\right)}
-\nabla_{p}J\nabla_{p}\left(\Delta J\right)^{2}\\
& =\Delta\left(J\left(\Delta J\right)^{2}\right)
-\underline{\nabla_{p}\left(\nabla_{p}J\left(\Delta J\right)^{2}\right)}
+\left(\Delta J\right)^{3}
+{T_{\lambda_0}}\\
& =\Delta\left(J\left(\Delta J\right)^{2}\right)
+\left(\Delta J\right)^{3}+{T_{\lambda_0}},
\end{align*}
we have
\begin{align}
\mathbf{IV}
& =J\Delta\left(\Delta J\right)^{2}+\Delta\left(\Delta J\right)^{2}J\nonumber \\
& =\Delta\bigg(J\left(\Delta J\right)^{2}+\left(\Delta J\right)^{2}J\bigg)
+2\left(\Delta J\right)^{3}
+{T_{\lambda_0}}.\label{eqn:m3:IV}
\end{align}

\bigskip
\noindent \textbf{Step Five: divergence forms of nonlinearity}
\medskip

Combining the equalities (\ref{eqn:m3:key}), (\ref{eqn:m3:II}), (\ref{eqn:m3:III})
and (\ref{eqn:m3:IV}), we derive that
\[
2\mathbf{I}+\mathbf{II}+2\mathbf{III}+2\mathbf{IV}+2\mathbf{V}=T_{\lambda_0},
\]

which completes the proof.

\end{document}